\documentclass[a4paper,11pt]{amsart}

\usepackage{amsmath}
\usepackage{amssymb,amsbsy,amsmath,amsfonts,amssymb,amscd}
\usepackage{latexsym}
\usepackage{amsthm}
\usepackage{graphics}
\usepackage{color}
\usepackage{tikz}
\usetikzlibrary{arrows}
\usepackage{tikz-cd}
\usepackage{pdfpages} 
\usepackage{longtable}
\usepackage{xy}
\usepackage{chngcntr}
\usepackage{array}
\usepackage{color}
\usepackage{comment}
\usepackage{enumitem}
\usepackage{geometry}
\usepackage{setspace}
\usepackage{multirow}
\usepackage{float}

\input xy
\xyoption{all}

\newcommand\Lam{\Lambda}

\newcommand\ze{\zeta}

\DeclareMathOperator{\Fix}{Fix}

\DeclareMathOperator{\Hom}{Hom}

\DeclareMathOperator{\Sym}{Sym}

\DeclareMathOperator{\Def}{Def}

\DeclareMathOperator{\Stab}{Stab}

\DeclareMathOperator{\PGL}{PGL}
\DeclareMathOperator{\AGL}{AGL}

\newcommand{\CC}{\ensuremath{\mathbb{C}}}
\newcommand{\RR}{\ensuremath{\mathbb{R}}}
\newcommand{\ZZ}{\ensuremath{\mathbb{Z}}}

\renewcommand{\Phi}{\phi}

\def\eea{\end{eqnarray*}}
\def\bea{\begin{eqnarray*}}

\DeclareMathOperator{\Aut}{Aut}
\DeclareMathOperator{\End}{End}

\DeclareMathOperator{\diag}{diag}

\DeclareMathOperator{\ord}{ord}

\DeclareMathOperator{\triv}{triv}

\newcommand\dual{\mathrel{\raise3pt\hbox{$\underline{\mathrm{\thinspace d
\thinspace}}$}}}
\newcommand\qe{\ifhmode\unskip\nobreak\fi\quad $\Box$}       

\def\BOX{\hfill\lower.5\baselineskip\hbox{$\Box$}}

\newtheorem{theorem}{Theorem}

\newtheorem{remark}[theorem]{Remark}
\newenvironment{rem}{\begin{remark}\rm}{\end{remark}}

\newtheorem{prop}[theorem]{Proposition}
\newtheorem{cor}[theorem]{Corollary}
\newtheorem{lemma}[theorem]{Lemma}
\newtheorem{example}[theorem]{Example}

\newtheorem{assump}[theorem]{Assumption}
\newtheorem{notation}[theorem]{Notation}

\numberwithin{theorem}{section}
\numberwithin{equation}{section}

\theoremstyle{definition}
\newtheorem{defin}[theorem]{Definition}
\newenvironment{definition}{\begin{defin}\rm}{\end{defin}}

\DeclareMathOperator{\Aff}{Aff}
\DeclareMathOperator{\im}{im}

\DeclareMathOperator{\GL}{GL}

\DeclareMathOperator{\id}{id}
\DeclareMathOperator{\Heis}{He}
\DeclareMathOperator{\diff}{diff}
\DeclareMathOperator{\exc}{exc.}
\usepackage[pagebackref=true]{hyperref}

\makeatletter
\def\tagform@#1{\maketag@@@{\ignorespaces#1\unskip\@@italiccorr}}
\makeatother

\newcolumntype{H}{@{}>{\lrbox0}l<{\endlrbox}} 
\newcommand{\mylabel}[2]{#2\def\@currentlabel{#2}\label{#1}}

\hypersetup{colorlinks=true, unicode=true, linkcolor=[rgb]{0.10,0.05,0.67}, citecolor=[rgb]{0.10,0.05,0.67}, filecolor=[rgb]{0.10,0.05,0.67}, urlcolor=[rgb]{0.10,0.05,0.67}}

\begin{document}

\title[Rigid Hyperelliptic Fourfolds]{The Classification of Rigid Hyperelliptic Fourfolds}
\author{Andreas Demleitner and Christian Glei\ss ner}
\address{Andreas Demleitner, University of Freiburg, Ernst-Zermelo-Str. 1,
D-79104 Freiburg, Germany \newline
Christian Gleissner, University of Bayreuth, Universit\"atsstr. 30, D-95447 Bayreuth, Germany}
\email{andreas.demleitner@math.uni-freiburg.de, \quad christian.gleissner@uni-bayreuth.de}

\thanks{
\textit{Acknowledgements:} The authors would like to thank Ingrid Bauer, Fabrizio Catanese, Julia Kotonski, Rafa\l~ Lutowski and Andrzej Szczepa\'nski
 for useful comments and discussions. Moreover, they thank an anonymous referee for a plethora of helpful suggestions to improve the manuscript, especially for significantly shortening the proof of Proposition \ref{propExcludeK}.
}

\begin{abstract}
We provide  a fine classification of rigid hyperelliptic manifolds in dimension four up to biholomorphism and diffeomorphism.   
These manifolds are explicitly described as finite \'etale quotients of a product of four Fermat elliptic curves. 
\end{abstract}

\maketitle

\tableofcontents

\section{Introduction}

It is classically known that any compact  flat Riemannian manifold $X$
is a  quotient of  the affine space $\mathbb R^n$ by  a torsion-free, discrete and cocompact group $\Gamma \leq \mathbb E(n)$ of Euclidean motions. 
Such groups $\Gamma$ are called Bieberbach groups. By Bieberbach's first structure theorem, there is a short exact sequence 
\[
0 \to \Lambda \to \Gamma \to G \to 1, 
\]
where the subgroup $\Lambda$ of translations is a lattice of full rank  
 and $G$ is a finite group isomorphic to  the holonomy  group of $X$. 
 In particular, $X$ is the quotient of the torus $T=\mathbb R^n/\Lambda$ by  the induced  action of  $G$. 
Obviously, not every flat manifold  $X=\mathbb R^{2n}/\Gamma$ has a complex  K\"ahler structure, which means that the holonomy representation is unitary. 
In case of existence, the manifold $X$ is called a {\em flat K\"ahler manifold} or a {\em (generalized) hyperelliptic manifold}. They have been classified in the surface case, i.e., in complex dimension $2$
by Bagnera-de Franchis \cite{bdf} as well as Enriques-Severi \cite{Enr-Sev}. In complex dimension $3$, these manifolds have been investigated in works of Uchida-Yoshihara \cite{Uchida-Yoshihara}, Lange \cite{Lange} and Catanese-Demleitner \cite{CD-2}. 
In \cite{Demleitner-thesis}, Demleitner derived a complete list of holonomy groups of generalized hyperelliptic fourfolds. At the moment, a fine classification up to biholomorphism  is not yet established. As a first step towards such a classification, we  restrict our attention to those  examples having  a rigid complex structure. \\
From the $\mathcal C^{\infty}$-point of view, compact flat Riemannian manifolds  behave very nicely. Indeed, 
as a consequence of  Bieberbach's second theorem, the diffeomorphism type of a compact flat Riemannian manifold is uniquely determined by the  fundamental group of the underlying topological space.
The example of complex tori shows that we cannot expect the analogous property in the holomorphic category without any additional assumptions. In fact,  in this paper we present  diffeomorphic but non-biholomorphic rigid hyperelliptic fourfolds, hence, this property fails even under the strong assumption of rigidity. 
We want to point out that Bieberbach's third theorem shows that in any dimension, there are only finitely many compact flat Riemannian manifolds  up to diffeomorphism. 
Our first main result is: 

\begin{theorem} \label{theorem-rigid-hyperelliptic}
	Let $X = T/G$ be a rigid hyperelliptic manifold with holonomy $G$, then:
	\begin{itemize}
		\item[(a)] $\dim(X) \geq 4$.
		\item[(b)] If $\dim(X) = 4$, then $G\cong \ZZ_3^2$ or $G \cong \Heis(3)$, the Heisenberg group of order $27$: 
\begin{align} \label{present-he3}
\Heis(3) := \langle g,h,k \ | \ g^3 = h^3 = k^3 = [g,k] = [h,k] = 1, ~ [g,h] = k\rangle.
\end{align}
	\end{itemize}
\end{theorem}

Our second main result is the full classification of rigid hyperelliptic fourfolds, up to biholomorphism and diffeomorphism: 

\begin{theorem}\label{Mani}
Let $E := \mathbb C/\mathbb Z[\zeta_3]$ and $t := (1+2\zeta_3)/3 \in E[3]$. Then:
\begin{enumerate}[ref=(\theenumi)]
	\item There are exactly twelve biholomorphism classes $X_i$ of rigid hyperelliptic fourfolds with holonomy $\mathbb Z_3^2$. They are realized as quotients of $E^4/K_i$ by the actions
	\begin{align*}
		\phi_i(a,b)(z) := \diag(\zeta_3^a, \ \zeta_3^b, \ \zeta_3^{2a+b}, \ \zeta_3^{a+b})\cdot z + \tau_i(a,b),
	\end{align*}
	where $K_i$ and $\tau_i$ are according to the table below:
	\begin{center}
		\bgroup\def\arraystretch{1.3}\begin{tabular}{|c|l|l|l|} \hline 
		$i$ & $K_i$ & $\tau_i(1,0)$ & $\tau_i(0,1)$ \\ \hline \hline 
		$1$ & $\{0\}$ & $(0,t,t,t)$ &  $(t,0,0,0)$ \\ \hline
		$2$ & $\langle (0,0,t,t)\rangle$ & $(0,t,t,t)$ &  $(t,0,0,0)$ \\  
		$3$ & $\langle (0,t,t,t)\rangle$ & $(0,1/3,1/3,1/3)$ &  $(t,0,0,0)$ \\  
		$4$ & $\langle (t,0,0,t) \rangle$ & $(0,t,t,1/3)$ &  $(2/3,0,0,0)$ \\  
		$5$ & $\langle (t,0,t,t)\rangle$ & $(0,t,1/3,1/3)$ &  $(2/3,0,0,0)$ \\  
		$6$ & $\langle (t,t,t,t) \rangle$ & $(0,1/3,1/3,1/3)$ &  $(2/3,0,0,0)$ \\ \hline
		$7$ & $\langle (0,t,t,t), \ (0,t,-t,0)\rangle$ & $(0,1/3,1/3,1/3)$ &  $(t,0,0,0)$ \\  
		$8$ & $\langle(0,0,t,t),\ (t,0,-t,0)\rangle$ & $(0,t,1/3,2/3)$ &  $(2/3,0,0,0)$ \\  
		$9$ & $\langle (t,t,0,0),\ (0,0,t,t)\rangle$ & $(0,1/3,1/3,1/3)$ &  $(2/3,0,0,0)$  \\
		$10$ & $\langle (0,0,t,t),\ (t,t,0,t) \rangle$ & $(0,1/3,1/3,2/3)$ &  $(2/3,0,0,0)$  \\
		$11$ & $\langle (t,0,0,t),\ (t,t,t,-t) \rangle$ & $(0,1/3,1/3,2/3)$ &  $(2/3,0,0,0)$  \\ \hline 
		$12$ & $\langle (-t,t,0,0),\ (t,0,t,t),\ (t,t,t,0) \rangle$ & $(0,1/3,1/3,2/3)$ &  $(1/3,0,0,0)$ \\ \hline 
		\end{tabular}
		\egroup
	\end{center}

	\medskip
	\noindent These twelve complex manifolds form eight diffeomorphism classes:
	\begin{align*}
		X_1, ~ \ X_2 \simeq_{\diff} X_4, ~\ X_3 \simeq_{\diff} X_5, ~\ X_6, ~\ X_7 \simeq_{\diff} X_8, ~\ X_9, ~\ X_{10} \simeq_{\diff} X_{11}, ~\ X_{12}.
	\end{align*}
Manifolds belonging to different $\mathcal C^\infty$-classes have non-isomorphic fundamental groups. 
A hyperelliptic fourfold  $X$  whose fundamental group is isomorphic to the fundamental group of a rigid hyperelliptic fourfold  with holonomy $\mathbb Z_3^2$ is rigid  and therefore biholomorphic to some $X_j$.

\bigskip
	\item \label{Mani2} There are exactly four biholomorphism classes $X_{i,j}$ of hyperelliptic fourfolds with holonomy $\Heis(3)$. They are realized as $E^4/K_i$ by the actions 
	\begin{align*}
	&\phi_{i,j}(g)(z) := \begin{pmatrix}
	1 & 0 & 0 & 0 \\
	0 & 0 & 0 & 1 \\
	0 & 1 & 0 & 0 \\
	0 & 0 & 1 & 0 
	\end{pmatrix}\cdot z + \tau_j(g) \quad \makebox{and} \quad \phi_{i,j}(h)(z) := \begin{pmatrix}
	\zeta_3 &&& \\ & 1 && \\ && \zeta_3^2 & \\ &&& \zeta_3
	\end{pmatrix} \cdot z + \tau_j(h),
	\end{align*}
	
	\bigskip
	\noindent 
	where $K_1 :=  \langle (0,t,t,t)\rangle$, $K_2 := \langle (0,t,t,t), (0,t,-t,0)\rangle$	and 
	\begin{center}
		\begin{tabular}{ll}
		$\tau_1(g) := (1/3,0,0,0),$ &   $\tau_1(h) := (0,1/3,1/3,1/3)$, \\
		$\tau_2(g) := (1/3,0,0,-t),$ &  $\tau_2(h) := (0,1/3,1/3,1/3)$.
	\end{tabular}
	\end{center}	
	The manifolds $X_{i,j}$ have pairwise distinct fundamental groups. Each hyperelliptic manifold with holonomy $\Heis(3)$ is rigid and therefore biholomorphic to 
	one of the  $X_{i,j}$.
	\end{enumerate} 
\end{theorem}

Observe that all manifolds in the theorem are finite quotients of a product of elliptic curves and therefore 
projective. The projectivity is not a coincidence, in fact, rigid hyperelliptic fourfolds cannot have global non-zero holomorphic 2-forms. Thus, 
$H^2(X,\mathbb Z)\otimes_{\mathbb Z} \mathbb C \simeq H^{1,1}(X)$ and consequently, there exist K\"ahler classes represented by  positive line bundles. 
More generally,  in \cite{CD}, it is shown that a hyperelliptic manifold has arbitrary small algebraic approximations. 

Some of the manifolds of Theorem \ref{Mani} were already discussed in the literature. In  \cite[Section 10.1.4]{Demleitner-thesis}, the author already presented the rigid hyperelliptic fourfold $X_{1,2}$ with holonomy $\Heis(3)$.
 In  \cite[Theorem 3.4]{Bauer-Catanese-rigid}, the reader can find for each $n\geq 4$  an 
example of a rigid hyperelliptic manifold with holonomy $\mathbb Z_3^2$. 
For $n=4$, their example is  $X_1$ from our Theorem \ref{Mani}. The same fourfold was found and discussed in 
 \cite[Section 5.1]{HalendaLutowski}. In  \cite[Theorem 5.4]{Bauer-Gleissner-2}, two rigid hyperelliptic 
 fourfolds with holonomy $\mathbb Z_3^2$ and different fundamental groups are  constructed.   In our classification, these examples are 
 $X_1$ and $X_3$. We point out that our manifolds have $b_1 = 0$: such manifolds are of independent interest and were studied in the papers \cite{HillerSah1} and \cite{HillerSah2}, where the authors give $X_1$ as an example. 

%
%
%
%
%
%
%
%
%
%
We will now sketch the outline of the paper. 
In Section \ref{section-Prelim}, we collect some  preliminaries concerning hyperelliptic manifolds and   explain the necessary tools  from deformation theory. 
 In particular, we show that the rigidity of the hyperelliptic manifold is encoded in the complex holonomy representation. 
 Section \ref{section-1stMain} is devoted to prove the  first main result of our paper, Theorem \ref{theorem-rigid-hyperelliptic}. 
 Going  through 
 the list of complex holonomy groups \cite{Uchida-Yoshihara}  
  in dimension $3$  and analyzing  their representation theory, we prove that none of these groups allow a rigid and free action. In dimension four, we use Demleitner's list of $79$ complex holonomy groups to  show that only the two groups $\mathbb Z_3^2$ and  $\Heis(3)$ allow a rigid action. 
  Moreover,  for both of these groups, there is, up to equivalence and automorphism,  a unique candidate for the complex holonomy representation $\rho$. 
In the third section, we recall Bieberbach's structure theorems of crystallographic and Bieberbach groups and 
 explain their geometric consequences in our setting. 
In the fourth section, we determine all lattices $\Lambda$ which have a  $\Heis(3)$ or $\mathbb Z_3^2$  module structure via the holonomy representation 
 $\rho$. Moreover, we determine all free and rigid actions of our holonomy groups on the tori $\mathbb C^4/\Lambda$. The linear parts of these actions are 
 given by $\rho$, while the translation parts are so called \emph{special}  cohomology classes in the group cohomology 
 $H^1(G,\mathbb C^4/\Lambda)$. 
In Section \ref{section-Bihol}, we decide, using the developed theory about Bieberbach groups, which fourfolds found in the previous chapter are biholomorphic or diffeomorphic, respectively. This amounts to determine the orbits of the action of a certain group on the special cohomology classes of  $H^1(G,\mathbb C^4/\Lambda)$.  To list the actions and determine these orbits, we use the computer algebra system MAGMA \cite{MAGMA}, which allows for an efficient
computation. The interested reader can find our code on the website:
\begin{center}
\url{http://www.staff.uni-bayreuth.de/~bt300503/publi.html}.
\end{center}

Finally, in the last section, we summarize the proof of Theorem \ref{Mani}.

\bigskip
\bigskip

\textbf{Notation.}  We use the standard notation from complex geometry and representation theory of finite groups. 
The group of affine linear transformations of $\mathbb K^n$ is denoted by $\AGL(n,\mathbb K)$. We write $\Aut(T)$ for the group of biholomorphic automorphisms of a complex torus $T$, whereas $\Aut_0(T)$ is the subgroup of group automorphisms. Similarly,  $\Aff(T)$ is the group of affine  diffeomorphisms of $T$ and $\Aff_0(T)$ the subgroup of diffeomorphisms fixing the origin.

\section{Basic Definitions and  Preliminaries} \label{section-Prelim}

In this section, we collect preliminaries concerning rigidity and hyperelliptic manifolds that we will use in our paper. \\
Let $T = V/\Lam$ be a complex torus. Since holomorphic maps between complex tori are affine, an automorphism $g \in \Aut(T)$ can be decomposed into its \textit{linear part} $\rho(g)$ and its \textit{translation part} $\tau(g)$, i.e., $g(z) = \rho(g)z + \tau(g)$. A subgroup $G \leq \Aut(T)$ therefore defines two representations
\begin{align*}
\rho \colon G \to \GL(V) \qquad \text{ and } \qquad \rho_{\Lambda} \colon G \to \GL(\Lambda),
\end{align*}
which both map an element of $G$ to its linear part, viewed as automorphisms of $V$ and $\Lambda$, respectively. These representations are called the 
\textit{analytic} or \textit{complex holonomy} and the \textit{integral holonomy} representation.  The relation between the two representations comes from the Hodge decomposition $\Lambda \otimes_\ZZ \CC = V \oplus \overline{V}$, which implies that $\rho_{\Lambda} \otimes \CC$ is equivalent to $\rho \oplus \overline{\rho}$. In particular, the characteristic polynomial of $\rho(g) \oplus  \overline{\rho(g)}$ has integral coefficients for all $g \in G$.

\begin{rem} \label{no-translations}
	Let $G$ be a finite group of automorphisms of the complex torus $T$. Denote by $H$ the normal subgroup of $G$ consisting of the translations. Then, $T/H$ is again a complex torus, and the canonical map $(T/H)/(G/H) \to T/G$ is biholomorphic. For this reason, it suffices to study actions of finite groups on complex tori which do not contain any translations.
\end{rem}

\begin{defin}
	A \textit{hyperelliptic manifold} is a quotient $X = T/G$ of a complex torus $T$ by a finite, non-trivial group $G \leq \Aut(T)$ which acts freely on $T$ and does not contain any translations. The group $G$ is called the \emph{holonomy group} of $X$.
\end{defin}

\begin{rem}\label{basicAnalytic}
	Let $X = T/G$ be a hyperelliptic manifold with holonomy $G$. 
	\begin{enumerate}
		\item The associated analytic representation $\rho \colon G \to \GL(V)$ is faithful because  $G$ does not contain translations.   
		\item Since $g \in G \setminus \{\id_T\}$ acts freely on $T = V/\Lam$, the lift of the fixed point equation
		\begin{align*}
		(\rho(g) - \id_V)z = \lambda - \tau(g)
		\end{align*}
		has no solution in $z \in V$ and $\lambda \in \Lambda$. In particular, $1$ is an eigenvalue of $\rho(g)$.
	\end{enumerate}
	
\end{rem}

There are several notions of rigidity, see \cite[Definition 2.1]{Bauer-Catanese-rigid}. We recall only the parts which are relevant to us.  

\begin{defin} \label{rigidity}
Let $X$ be a compact complex manifold.
\begin{enumerate}
	\item A \textit{deformation} of $X$ consists of the following data:
	\begin{itemize}
		\item a flat and proper holomorphic map  $\pi \colon \mathfrak X \to B$ of connected complex spaces,
		\item a point $0 \in B$,
		\item an isomorphism $\pi^{-1}(\{0\}) \simeq X$. 
	\end{itemize}
	\item We call $X$ \textit{(locally) rigid} if for every deformation  $\pi \colon \mathfrak X \to B$ of $X$, there is an open neighborhood $U \subset B$ of $0$ such that $\pi^{-1}(U) \simeq X \times U$ and $\pi|_{\pi^{-1}(U)} \colon X \times U \to U$ is the projection onto the second factor.
	\item We call $X$ \textit{infinitesimally rigid} if $H^1(X,\Theta_X) = 0$, where $\Theta_X$ 
	is the holomorphic tangent bundle of $X$.
\end{enumerate}
\end{defin}

Kodaira-Spencer-Kuranishi theory shows that an infinitesimally rigid complex manifold is also rigid (see \cite{Catanese-Guide} for an account on deformation theory). 
It was a question of Kodaira and Morrow \cite[Problem on p. 45]{KM} whether the converse implication holds.  In general, this is false. 
The first counterexamples were given by 
Bauer and Pignatelli \cite{Bauer-Pignatelli}, see also the paper of B\"ohning,  Bothmer and Pignatelli \cite{BBP}.
However, for hyperelliptic manifolds, the two notions coincide: 

\begin{prop}\label{rigidinf}
A hyperelliptic manifold $X=T/G$ is rigid if and only if it is infinitesimally rigid. 
\end{prop}

\begin{proof}
According to  \cite[Section 5 and Proposition 3]{CD}, $X$ is rigid if and only if the pair $(T,G)$ is rigid, i.e., 
$\Def(T)^G:=\Def(T) \cap H^1(T,\Theta_{T})^G$ is a  point. 
 Since $\Def(T)$ is smooth of dimension $h^1(T,\Theta_T)$ (cf. \cite[p. 230 ff.]{Kodaira}), the set 
$\Def(T)^G$ consists of a single point if and only if $H^1(T,\Theta_{T})^G=0$. This precisely means that  $H^1(X,\Theta_{X})\simeq 
H^1(T,\Theta_{T})^G=0$ because the quotient map $\pi \colon T \to T/G$ is unramified. 
\end{proof}

The rigidity of a hyperelliptic manifold $X=T/G$ depends only on the associated analytic representation: 
\begin{cor}\label{ConjRig}
A hyperelliptic manifold $X=T/G$ is rigid if and only if the analytic representation $\rho\colon G \to \GL(V)$ and its conjugate 
$\overline{\rho}$ have no common irreducible subrepresentation. 
\end{cor}

\begin{proof}
Using Dolbeault's interpretation of cohomology, we write 
\[
H^1(T,\Theta_T) \simeq  H^{0,1}(\Theta_T) = \langle d\overline{z}_1,\ldots, d\overline{z}_n\rangle \otimes 
\bigg\langle \frac{\partial}{\partial z_1}, \ldots,  \frac{\partial}{\partial z_n} \bigg\rangle. 
\]
By definition, the $G$-action on $V=\big\langle \frac{\partial}{\partial z_1}, \ldots,  \frac{\partial}{\partial z_n} \big\rangle $ is the analytic representation,  while the action on $\langle d\overline{z}_1,\ldots, d\overline{z}_n\rangle $ is the dual of the complex conjugate of the analytic representation.   This gives  an isomorphism of representations 
$H^1(T,\Theta_T) \simeq \overline{V}^{\vee} \otimes V \simeq  \Hom(\overline{V},V)$ and the claim follows from Schur's lemma and Proposition \ref{rigidinf}. 
\end{proof}

\begin{rem}\
\begin{enumerate}
\item
The corollary tells us in particular that  the analytic representation of a rigid hyperelliptic manifold contains no self-conjugate irreducible representations, i.e., no representations of \emph{real or  quaternionic type} \cite[p. 108]{Serre}.  
\item
A rigid hyperelliptic manifold $X$ has no non-zero global holomorphic 2-forms because 
\[
0=\langle \overline{\chi_{\rho}}, \chi_{\rho} \rangle = \langle \overline{\chi_{\rho}}^2,\chi_{\triv} \rangle =  \langle \wedge^2 (\overline{\chi_{\rho}}),\chi_{\triv} \rangle + \langle \Sym^2 (\overline{\chi_{\rho}}),\chi_{\triv} \rangle
\] 
implies $h^0(X,\Omega_{X}^2)=0$.
\end{enumerate}
\end{rem}

\section{Proof of the First Main Theorem} \label{section-1stMain}

In order to prove Theorem \ref{theorem-rigid-hyperelliptic},  we use the classification of hyperelliptic groups according to Bagnera-de Franchis
\cite{bdf} in dimension 2, Uchida-Yoshihara \cite{Uchida-Yoshihara} in dimension $3$ (later completed by Catanese and Demleitner  \cite{CD-2}) and Demleitner  \cite{Demleitner-thesis} in dimension $4$. It would  be interesting to find a proof which does not depend on these  classification results. 

\begin{proof}[Proof of Theorem \ref{theorem-rigid-hyperelliptic} (a)]
Hyperelliptic surfaces are never rigid because their 
holonomy groups are cyclic: Remark \ref{basicAnalytic} (2) then implies that 
the associated analytic representation contains the trivial representation.  
To prove that there are no rigid  hyperelliptic threefolds, we use the list of groups of  Uchida-Yoshihara \cite{Uchida-Yoshihara}: the group is either the dihedral group $\mathcal D_4$ of order $8$ or an Abelian group of the form $\ZZ_{d_1} \times \ZZ_{d_2}$, where  $d_1|d_2$.  For the latter, we can restrict to the ones  with $d_1 > 1$, i.e., the  non-cyclic Abelian groups.  Since the representation theories of $\mathcal D_4$ and $\ZZ_2^2$ are real, no hyperelliptic manifold with holonomy group $\mathcal D_4$ or $\ZZ_2^2$ is rigid.  For $(d_1,d_2) \neq (2,2)$, $d_1 \neq 1$, the assertion follows from \cite[Lemma 6.5]{Lange} or  \cite[Theorem 6.1.8]{Demleitner-thesis},  which tells us  that the associated analytic representation contains the trivial representation. 
\end{proof}

\begin{proof}[Proof of Theorem \ref{theorem-rigid-hyperelliptic} (b)]
In analogy to part (a), we  use the classification of hyperelliptic groups in dimension $4$ achieved in \cite{Demleitner-thesis}.  For each of the $79$ holonomy groups $G$, we check
the existence of a representation $\rho \colon G \to \GL(4,\mathbb C)$ with the following properties:
\begin{enumerate}
\item $\rho$ is faithful  (Remark \ref{basicAnalytic} (1)),
\item  each representation matrix $\rho(g)$ contains 1 as an eigenvalue (Remark \ref{basicAnalytic} (2)),
\item the characteristic polynomial of $\rho(g) \oplus \overline{\rho(g)}$ is in $\mathbb Z[x]$ for all $g\in G$ (integral representation),
\item the character $\chi$  of $\rho$ and its conjugate $\overline{\chi}$ do not contain common irreducible characters 
(Corollary \ref{ConjRig}).
\end{enumerate}
To verify the existence of such a representation,  only the character table of $G$ is needed. Clearly, the kernel of a 
representation $\rho$ is equal to the kernel of its character: 
\[
\ker(\chi):= \lbrace g \in G ~ | ~ \chi(g)=\chi(1) \rbrace.
\]
Moreover, it is well-known that the characteristic polynomial $\rho(g)$ can be determined from the character values $\chi(g^i)$ 
thanks to   the Newton identities. 
We use MAGMA to run through the 
$79$ groups and check if there is a representation $\rho$ satisfying  conditions  (1) -- (4). We find that the only groups admitting such a representation are $\ZZ_3^2$ and $\Heis(3)$.
\end{proof}

\begin{rem}\label{ComplRepHe}
We recall that $\Heis(3)$ has exactly two non-equivalent irreducible representations of dimension $3$: the representation $\rho_3$, given by
\begin{align}\label{rho3}
\rho_3(g) = \begin{pmatrix}
	0 & 0 & 1 \\ 1 & 0 & 0 \\ 0 & 1 & 0 
	\end{pmatrix}, \qquad \rho_3(h) = \begin{pmatrix}
	1 & & \\ & \zeta_3^2 & \\ && \zeta_3
	\end{pmatrix}, \qquad \rho_3(k) = \begin{pmatrix}
	\zeta_3 && \\ & \zeta_3 & \\ && \zeta_3
	\end{pmatrix}
\end{align}
and its complex conjugate $\overline{\rho_3}$. Moreover, there are nine $1$-dimensional representations obtained from the central quotient $\mathbb Z_3^2$ 
by inflation. 
\end{rem}

\begin{rem}\label{analyticreps}
Analyzing  the output of the MAGMA computation in the proof of Theorem \ref{theorem-rigid-hyperelliptic} (b), we see that after application of  a suitable automorphism of $G$,  the analytic representation of a rigid hyperelliptic fourfold $X=T/G$ is equivalent to:  
\begin{itemize}
\item
$\rho(a,b) =\diag(\zeta_3^a, \ \zeta_3^b, \ \zeta_3^{2a+b}, \ \zeta_3^{a+b})$  if $G=\mathbb Z_3^2$ and
\medskip
\item 
$\rho=\rho_1\oplus\rho_3$ if  $G=\Heis(3)$.  Here, $\rho_1$ is  defined by  $\rho_1(g):=1$ and $\rho_1(h):=\zeta_3$. 
\end{itemize}
\end{rem}

As a consequence of Remark \ref{analyticreps}, all rigid hyperelliptic fourfolds with a fixed holonomy group have the same Hodge numbers:  

\begin{cor}
	The Hodge numbers of a rigid hyperelliptic fourfold  $X=T/G$ are:  
	\medskip
	\begin{center}
		\begin{tabular}{ccccc} 
			$\begin{matrix}
			&   &  &  & 1 &  &  &  &  \\
			&   &  & 0 &  & 0 &  &  &  \\
			&   & 0 &  & 2 &  & 0  &  & \\
			&  1 &  & 1 &   & 1 &   &   1 &  \\
			0&    & 0 &  & 2  &  & 0  &     &  0 \\
			\end{matrix}$  
			& 	& and	& &
			
			$\begin{matrix}
			&   &  &  & 1 &  &  &  &  \\
			&   &  & 0 &  & 0 &  &  &  \\
			&   & 0 &  & 4 &  & 0  &  & \\
			&  1 &  & 3 &   & 3 &   &   1 &  \\
			0&    & 0 &  & 6  &  & 0  &     &  0 \\
			\end{matrix}$ \\ 
			\\ 
			{\scriptsize $G=\Heis(3)$} & & & & {\scriptsize $G=\mathbb Z_3^2$}  
		\end{tabular}
	\end{center}
\end{cor}

\begin{proof}
	We consider the case $G=\Heis(3)$ first. 
	Since the action of $\Heis(3)$ on $T$ is free,  the Hodge numbers of $X$  are given as 
	\[
	h^{p,q}(X)=\dim_{\mathbb C}(H^{p,q}(T)^{\Heis(3)}).  
	\]
	They are the multiplicities  of the trivial representation  in 
	\[
	\psi_{p,q} \colon \Heis(3) \to \GL\big(H^{p,q}(T)\big), \qquad u \mapsto [\omega \mapsto \rho(u^{-1})^{\ast} \omega]. 
	\]
	According to Remark \ref{analyticreps}, we may assume that  
	$\rho=\rho_1\oplus \rho_3$. 
	Then the character of  $\psi_{p,q}$ is given by: 
	\[
	\chi_{p,q} = \sum_{\substack {s_1+s_2 =p \\ t_1 + t_2  =q}} \wedge^{s_1}(\overline{\chi_1}) \wedge^{t_1}(\chi_1) \cdot \wedge^{s_2} (\overline{\chi_3}) 
	\wedge^{t_2} (\chi_3), \qquad \makebox{where \quad  $\chi_i:=$ character of  $\rho_i$}.
	\]
	Now, the Hodge numbers $h^{p,q}(X)= \langle \chi_{p,q}, \chi_{\triv} \rangle$  can be  easily computed using the formulae for 
	$\chi_{p,q}$
	and  the identities
	\[
	\wedge^3(\chi_3)=\det(\rho_3)= \chi_{\triv} \quad \makebox{and} \quad  \wedge^2(\chi_3)=\overline{\chi_3}.
	\]
	To determine the Hodge numbers  $h^{p,q}$  in the  $G=\mathbb Z_3^2$ case, we restrict the characters $ \chi_{p,q} $ to the subgroup $\langle h,k \rangle \leq \Heis(3)$, which we identify with $\mathbb Z_3^2$ by the isomorphism $h \mapsto (1,0)$, $k \mapsto (0,1)$. We then compute the inner product of these characters with the trivial character of $\mathbb Z_3^2$. 
\end{proof}

\begin{prop} \label{He3-rigid}
	Any hyperelliptic fourfold with holonomy group $\Heis(3)$ is rigid. 
\end{prop}
\begin{proof}
	Since the associated analytic representation $\rho \colon \Heis(3) \to \GL(4,\mathbb C)$ is faithful, it decomposes in a one dimensional and an irreducible $3$-dimensional representation. The freeness of   the action  implies that the  $1$-dimensional summand is not trivial,  \cite[Proposition 10.1.21]{Demleitner-thesis}. Corollary \ref{ConjRig} implies the rigidity. 
\end{proof}

\section{Bieberbach's Structure Theorems  and their Geometric Consequences} \label{bieberer-section}

In this section, we consider hyperelliptic manifolds from the differential geometric point of view as compact flat Riemannian manifolds
or, equivalently,  as quotients of Euclidean spaces  by Bieberbach groups. 
Using Bieberbach's  structure  theorems, we translate  the problem of biholomorphic, diffeomorphic  and  homeomorphic classification of hyperelliptic manifolds into a group theoretical problem. This will play a crucial role in the proof of our main theorem because in our setting,  the latter  can be solved algorithmically.  

\begin{rem}\label{FundBie} \
\begin{enumerate}[ref=(\theenumi)]
\item \label{FundBie1} In the sequel, we often identify $\mathbb R^{2n}$ with $\mathbb C^n$ via 
\[
\qquad (x_1, y_1, \ldots, x_n,y_n) \mapsto (z_1, \ldots, z_n), \quad \makebox{where} \quad z_i = x_i + \sqrt{-1}y_i. 
\]
In this way we can view any 
complex representation $\rho \colon G \to \GL(n,\mathbb C)$ as a real representation $\rho_{\mathbb R} \colon G \to \GL(2n,\mathbb R)$. 
Over $\mathbb C$, the representation  $\rho_{\mathbb R}$ decomposes into $\rho \oplus \overline{\rho}$.  
\item 
A matrix $A\in \GL(2n,\mathbb R)$ induces, under the  identification in (1), a bijection  $f_A\colon \mathbb C^n \to \mathbb C^n$ which is in general only $\mathbb R$-linear. 
We say that the matrix $A$ is $\mathbb C$-linear, if $f_A$ is $\mathbb C$-linear, and $\mathbb C$-antilinear if $f_A$ is 
$\mathbb C$-antilinear, i.e., $f_A(\lambda v)=\overline{\lambda}f_A(v)$ for all $\lambda \in \mathbb C$ and $v \in \mathbb C^n$. 
\item
The fundamental group of a hyperelliptic manifold $X=T/G$ is isomorphic to  
the group of deck  transformations $\Gamma$ of the universal cover 
\[
\mathbb C^n \to T \to X.
\]
It  consists of all lifts of the elements of  $G$ to $\mathbb C^n$ and is therefore a cocompact, free and discrete group of affine transformations. 
More precisely, since $G$ is finite, we may assume that the analytic representation  $\rho$ is unitary.  Then, we can consider  $\Gamma$ as a  subgroup of 
$\mathbb C^n \rtimes U(n)$. 
The identification in (1) 
allows us to view $\rho$ as a real representation $\rho_{\mathbb R} \colon G \to O(2n)$ and $\Gamma$ as a subgroup of the group of Euclidean motions $\mathbb E(2n)$. 
As   the action of $G$ on $T$ does not contain translations, the lattice $\Lambda$ of the torus $T$ is equal to the intersection  
$\Gamma \cap \mathbb C^n$, which is the translation subgroup of $\Gamma$. 
\end{enumerate}
\end{rem}

\begin{definition}
A discrete cocompact subgroup of $\mathbb{E}(n)$ is called a {\em crystallographic group}.  
A {\em Bieberbach group} is a torsion free crystallographic group. 
\end{definition}

\noindent
The above remark tells us that the fundamental group of a rigid  hyperelliptic  manifold  with holonomy group $G$ is a Bieberbach group and, with the above notation,
$\Gamma/\Lambda \simeq G$. 

\begin{rem}
Given a Bieberbach group $\Gamma \leq \mathbb E(n)$, 
the quotient $X=\mathbb R^n/\Gamma$ is a differentiable manifold. It inherits a natural Riemannian metric,  induced by the Euclidean inner product, and is therefore flat, i.e.,  locally isometric to $\mathbb R^n$. Conversely, any compact flat Riemannian manifold  is isometric to a quotient of $\mathbb R^n$ by a Bieberbach group. 
\end{rem}

In the sequel, we will use  the classical structure theorems of crystallographic and  Bieberbach groups, see \cite[Chapter I]{Charlap} and \cite[Theorem 2.1]{Szczepanski}. We would like to remark that  the classical reference \cite{Charlap} applies the theory to compact flat Riemannian manifolds only, whereas the book \cite{Szczepanski}  studies, in addition to other recent developments, also hyperelliptic and flat K\"ahler manifolds, respectively.

\begin{theorem}[Bieberbach's structure theorems]\
\begin{enumerate}
\item The translation subgroup $\Lambda:=\Gamma \cap \mathbb R^n$ of a crystallographic group  $\Gamma \leq \mathbb{E}(n)$ is a 
lattice  of rank $n$ and $\Gamma/\Lambda$ is finite.  All  other normal Abelian subgroups of $\Gamma$ are contained in $\Lambda$.

\item Let  $\Gamma, \Gamma' \leq \mathbb{E}(n)$ be two crystallographic groups and $\gamma \colon \Gamma \to \Gamma'$ be an isomorphism. Then, there  exists 
an affine transformation  
 $\alpha \in \AGL(n,\mathbb R)$ such that $\gamma(g)=\alpha \circ g \circ \alpha^{-1}$ for all $g\in \Gamma$.  

\item In each dimension,  there are only finitely many isomorphism classes of crystallographic groups. 
\end{enumerate}
\end{theorem}

\noindent
In our setting, these theorems have the  following  consequences (cf. \cite[Section 3]{HalendaLutowski}):

\begin{cor}\label{GeometricBieber} 
Let $\Phi \colon G \to \Aut(T)$ and $\Phi' \colon G' \to \Aut(T')$  be  free and translation free holomorphic actions of finite groups $G$ and $G'$. Assume there is an isomorphism 
$\gamma \colon \Gamma \to \Gamma'$ between the fundamental groups of  $X=T/G$ and $X'=T'/G'$.
\begin{enumerate}
\item
We view the groups $\Gamma$ and $\Gamma'$ as Bieberbach groups in $\mathbb E(2n)$.  Thanks to  Bieberbach's  first  theorem, $\gamma$ restricts to an isomorphism 
$\gamma \colon \Lambda \to \Lambda'$, and it follows that $G$ and $G'$ are isomorphic.   
\item According to Bieberbach's  second theorem, $\gamma$ is  the  conjugation by an 
affinity
\[
\alpha(x)=Ax+d, \qquad A \in \GL(2n,\mathbb R), ~~ d \in \mathbb R^{2n}. 
\]
In particular, $\Lambda'=A\cdot \Lambda$.
The affinity 
$\alpha$ induces  diffeomorphisms $\widehat{\alpha}$ and $\widetilde{\alpha}$, 
such that the following diagram  commutes: 
\[
\xymatrix{
	T   \ar[d]\ar[r]^{\widetilde{\alpha}} & T'  \ar[d] \\
	X \ar[r]^{\widehat{\alpha}} & X'.}
\]
This holds in particular for $\gamma=f_{\ast}$, where $f\colon X_1 \to X_2$ is a homeomorphism. 
%
%
\item
There is an isomorphism  $\varphi \colon G \to G'$ such that 
\[
A\Phi(u)(x)+d=\Phi'\big(\varphi(u)\big)(Ax+d) \qquad \makebox{for all}  \quad u   \in G \quad \makebox{and} \quad  x \in T.  
\]
This comes from the fact that $\Phi(G)$ and $\phi'(G')$ are conjugated by $\widetilde{\alpha}$.
\item If $f \colon X \to X'$  is biholomorphic, then it lifts to a biholomorphism of the tori, i.e., it is induced by 
 an affinity 
 \[
 \alpha(x)=Ax+d, \qquad \makebox{where} \qquad A \in \GL(n,\mathbb C).
\]
\end{enumerate}
\end{cor}
As a consequence, in the category of hyperelliptic manifolds,  or more generally compact flat Riemannian manifolds,  the 
 following equivalence relations coincide: 
\begin{center}
\emph{isomorphic fundamental groups, \quad homeomorphic, \quad diffeomorphic \quad} and \emph{\quad affine diffeomorphic.} 
\end{center}
In view of the classification problem, we may  restrict our attention to  affine diffeomorphisms and biholomorphic maps.  
According to Corollary \ref{GeometricBieber} (1),  it suffices to consider 
hyperelliptic manifolds $X$, $X'$ with the same holonomy group $G$.

\begin{rem}\label{Con1and2}
Let $f\colon X \to X'$ be a diffeomorphism induced by an affinity $\alpha(x)=Ax+d$. Writing  $\phi(u)(x)=\rho_{\mathbb R}(u)(x) + \tau(u)$ and similarly  $\phi'(u)(x)=\rho'_{\mathbb R}(u)(x) + \tau'(u)$, 
the equation in Corollary  \ref{GeometricBieber} (3) is equivalent to 
\[
(1) ~  A\rho_{\mathbb R}(u)A^{-1} =\rho'_{\mathbb R}(\varphi(u)) \qquad  \makebox{and}  \qquad (2) ~
(\rho'_{\mathbb R}(\varphi(u))-I_{2n})d = A\tau(u)-\tau'(\varphi(u)) 
\]
for all $u \in G$. 
Item $(1)$ means that $\rho_{\mathbb R}$ and $\rho'_{\mathbb R} \circ \varphi $ are equivalent as real representations. 
Item $(2)$ is an  equation holding in  the group  $T'$. 
If $f$ is holomorphic, i.e.,  $A$ is $\mathbb C$-linear, then $\rho$ and $\rho' \circ \varphi $ are equivalent as complex representations. 
\end{rem}

\begin{rem} \label{cocycle-rem}
Let $f\colon X \to X'$ be a diffeomorphism induced by $\alpha(x)=Ax+d$. 
\begin{enumerate}
\item 
Since $A\cdot \Lambda = \Lambda'$, condition $(1)$ of  Remark \ref{Con1and2}
 precisely means  that $A$ is contained in 
 \[
 \mathcal N_{\mathbb R}(\Lambda,\Lambda'):=\lbrace A \in \GL(2n,\mathbb R) ~ \big\vert ~ A \cdot \Lambda =\Lambda', ~  A\cdot \im(\rho_{\mathbb R})=\im(\rho'_{\mathbb R}) \cdot A \rbrace.
 \]
 If $f$ is biholomorphic, then $A$ is even contained in the subset 
 \[
 \mathcal N_{\mathbb C}(\Lambda,\Lambda') := \mathcal N_{\mathbb R}(\Lambda,\Lambda') \cap \GL(n,\mathbb C).
 \] 
\item 
In contrast to the linear parts, the  translation parts $\tau \colon G \to T$ and $\tau'\colon G\to T'$ of the actions $\Phi$ and $\Phi'$  are not homomorphisms, but
$1$-cocycles:
\[
\tau(u_1u_2) = \tau(u_1) + \rho(u_1)\tau(u_2), \quad \tau'(u_1u_2) = \tau'(u_1) + \rho'(u_1)\tau'(u_2)
\]
for all $u_1,u_2 \in G$.
Their classes  define  elements in the cohomology groups $H^1(G,T)$ and $H^1(G,T')$. Conversely, a class in 
$H^1(G,T)$ together with $\rho$  defines an action on $T$, which is unique up to conjugation. The class is called \emph{special} if the action is free. 
Condition (2) in  Remark \ref{Con1and2}  has the following cohomological interpretation: replacing $u$ with 
$\varphi^{-1}(u)$, it  reads 
\[
\rho'(u)d-d  = A\tau(\varphi^{-1}(u))-\tau'(u), 
\]
which  just means that the cocycles $A \tau \circ \varphi^{-1}$ and $\tau'$ belong to the same cohomology class. 
\item
In the special case where $\rho=\rho'$  and  $T=T'$, the sets $\mathcal N_{\mathbb R}(\Lambda,\Lambda)$ and 
$\mathcal N_{\mathbb C}(\Lambda,\Lambda)$ are the normalizers of $\im(\rho_{\mathbb R})$ in  $\Aff_0(T)$ and in  $\Aut_0(T)$. 
For simplicity, we denote them by $\mathcal N_{\mathbb R}(\Lambda)$ and $\mathcal N_{\mathbb C}(\Lambda)$. 
By (2), they  act on $H^1(G,T)$ by 
\[
A \ast \tau(u) := A\cdot \tau(\varphi_A^{-1}(u)),
\]
where  $\varphi_A$ is the unique automorphism of $G$ with $A\rho_{\mathbb R} A^{-1} =\rho_{\mathbb R}\circ \varphi_A$.
It follows that $A$ is the linear part of an affine diffeomorphism (or biholomorphism) $X \to X'$  if and only if $A\ast \tau$ and $\tau'$ represent the same  class in $H^1(G,T)$. 
\end{enumerate}
\end{rem}

The strategy to derive  of our main theorem is now clear:

\medskip
\noindent 
{\bf Scheme for the Classification}

\begin{enumerate}
\item \label{scheme-first}
In the first step, we  determine all lattices $\Lambda$ which have a  $G=\Heis(3)$ (or $\mathbb Z_3^2$)  module structure via the 
representation $\rho$ from  Remark \ref{analyticreps}. 
\item 
In the second step we determine for each $T=\mathbb C^4/\Lambda$ and for  all  special cohomology classes in  $H^1(G,T)$ a  
 representative $\tau$. Taking the quotients of $T$ by the actions $$\Phi(u)(z)=\rho(u)z+\tau(u),$$ we obtain 
all possible rigid hyperelliptic fourfolds. 
\item
In the final  step we decide which fourfolds found in the previous step are biholomorphic or diffeomorphic, respectively. 
\end{enumerate}

\begin{rem}\label{SeqRem}
Let $X=T/G$ be a hyperelliptic manifold, where $T=\mathbb C^n/\Lambda$ is a complex torus. The short exact sequence of $G$-modules 
\[
0 \to \Lambda \to \mathbb C^n \to T \to 0 
\]
 induces the long exact sequence
 \[
 \ldots \to H^1(G,\mathbb C^n) \to H^1(G,T) \to H^2(G,\Lambda) \to H^2(G,\mathbb C^n) \to \ldots
 \]
 in group cohomology.  By the vanishing of $H^i(G,\mathbb C^n)$ for $i\geq 1$, we obtain an isomorphism 
$ H^1(G,T) \simeq H^2(G,\Lambda)$.
 Hence, the $1$-cocycle  $\tau$ defining the translation part of the $G$-action on $T$ yields a class in 
 $H^2(G,\Lambda)$, which  corresponds to the extension 
 \[
 0 \to \Lambda \to \Gamma \to G \to 1, \qquad \makebox{where} \qquad \Gamma =\pi_1(X). 
 \]
 The biholomorphism or  diffeomorphism problem of hyperelliptic fourfolds can therefore be reinterpreted in terms of group extensions: 
 Let $X$ and $X'$ be two hyperelliptic manifolds with the same  holonomy group $G$ and let $\tau$ and $\tau'$ be the $2$-cocycles corresponding to 
the extensions 
\[
0 \to \Lambda \to \Gamma \to G \to 1 \qquad \makebox{and} \qquad 0 \to \Lambda' \to \Gamma' \to G \to 1.
\]
 Then, $X$ and $X'$ are biholomorphic (or diffeomorphic) if and only if there exists a matrix $A \in \mathcal N_{\mathbb C}(\Lambda,\Lambda')$
  (or $\mathcal N_{\mathbb R}(\Lambda,\Lambda')$) such that 
 $\tau'$ and $A\ast \tau$ belong to the same cohomology class of $H^2(G,\Lambda')$. Here, the $2$-cocycle $A\ast \tau$ is defined in the obvious way: 
 \[
 A \ast \tau(u_1,u_2) := A\tau(\varphi_A^{-1}(u_1),\varphi_A^{-1}(u_1)). 
 \]
 This amounts to say that $\Gamma$ and $\Gamma'$ are conjugated by an affinity $\alpha(x)=Ax+d$, for a suitable $d$. 
 The conjugation isomorphism induces an isomorphism of short exact sequences: 
 \[
\xymatrix{
0  \ar[r]   & \Lambda  \ar[r] \ar[d]  & \Gamma  \ar[r] \ar[d] & G \ar[r] \ar[d]& 1 \\
0  \ar[r]& \Lambda'  \ar[r]  & \Gamma' \ar[r] & G \ar[r] & 1} 
\] 
 
In view of the discussion above, our strategy parallels the  classification scheme for Bieberbach groups outlined in  \cite[Chapter III, Section 2]{Charlap} and \cite[Section 3.1, p. 30]{Szczepanski}, except for the last step. Here, for sake of the biholomorphic classification, we consider the refined  
equivalence relation of  conjugation of Bieberbach groups with holomorphic  affinities $\alpha(x)=Ax+d$,  i.e., those where  $A$ is $\mathbb C$-linear. 
\end{rem}

\section{Lattices and Cocycles Corresponding to Rigid Hyperelliptic Fourfolds}

We follow our classification strategy and first determine the possible lattices $\Lambda$. 
As a first step, we show that each candidate for $\Lambda$ contains $\mathbb Z[\zeta_3]^4$ as a sublattice of finite index. 

\begin{notation}
In the sequel, we will often consider $\mathbb Z_3^2$ as a subgroup of $\Heis(3)$ by identifying $(1,0)$ with $h$ and $(0,1)$ with $k$. Observe that this is compatible with the analytic representations given in Remark \ref{analyticreps}.
\end{notation}

\begin{prop} \label{isog}
Let $X=T/G$ be a rigid hyperelliptic fourfold, then the torus $T$ is equivariantly isogenous to a product of four elliptic curves $E_i \subset T$, each of which is a copy of the Fermat elliptic curve $\CC/\mathbb Z[\zeta_3]$. 
\end{prop}

\begin{proof}
According to Theorem  \ref{theorem-rigid-hyperelliptic}, the group is $G=\mathbb Z_3^2$ or $\Heis(3)$. Thanks to 
Remark \ref{analyticreps}, we can assume 
\[
\rho(h)=\diag(\zeta_3,\ 1, \ \zeta_3^2, \ \zeta_3) \quad \makebox{and} \quad \rho(k)=\diag(1,\ \zeta_3, \ \zeta_3, \ \zeta_3), 
\]
in both cases $G = \mathbb Z_3^2$ and $\Heis(3)$. 
Consider the following subtori 
\begin{itemize}
\item
$E_1 := \ker(\rho(k) - \id_T)^0$, 
\item
$E_2 := \ker(\rho(h)-\id_T)^0$, 
\item
$E_3 := \ker(\rho(hk)-\id_T)^0$ and  
\item
$E_4 := \ker(\rho(hk^2)-\id_T)^0$.
\end{itemize}
Here, the superscript $0$ denotes the connected component of the identity.
By construction, $\dim(E_i)=1$ and $\zeta_3 \in \Aut(E_i)$, therefore $E_i\simeq \CC/\mathbb Z[\zeta_3]$.
We conclude the proof, because the addition map 
\[
\mu \colon E_1 \times \ldots \times E_4 \to T 
\]
is a surjective holomorphic group homomorphism between tori of the same dimension, i.e., an isogeny (cf. \cite[p. 12]{Birkenhake-Lange}). 
The equivariance of $\mu$ holds  by construction. 
\end{proof}

Proposition \ref{isog} allows us to write $T \simeq (E_1 \times E_2 \times E_3 \times E_4)/K$,
	where $K$ is a finite group of translations. We will usually identify all $E_i$ with the Fermat elliptic curve $E:=\CC/\ZZ[\zeta_3]$ and only keep the indices if they are relevant for our arguments.


Let $T/G$ be a rigid hyperelliptic fourfold. From now on, we assume that $T = E^4/K$ (where $E = \mathbb C/\mathbb Z[\zeta_3]$ is the Fermat elliptic curve) and that the associated analytic representation $\rho$ is given as in Remark \ref{analyticreps}.

\begin{rem}\label{heis-action}
	Let $\Phi \colon G \hookrightarrow \Aut(T)$ be a faithful rigid holomorphic action. 
	\begin{itemize}
	\item
	If $G = \Heis(3)$, then, up to a change of origin in the elliptic curves, the translation part $\tau \colon G \to T$ of $\Phi$ can be written in the form
	\[
		\tau(g)= \left( a_1, \ a_2, \ a_3, \ a_4\right), \quad 
		\tau(h)= \left(0, \ b_2, \ b_3, \ b_4\right), \quad
		\tau(k)= \left(c_1, \ 0, \ 0, \ 0\right).
	\]
\item 
If $G=\mathbb Z_3^2=\langle h,k\rangle$, then the translation part of $\Phi$ can be written as 
	\[
		\tau(h)= \left(0, \ b_2, \ b_3, \ b_4\right), \quad
		\tau(k)= \left(c_1, \ 0, \ 0, \ 0\right).
	\]
\end{itemize}	
	We call an action with such a translation part an action in 
\emph{standard form}.  Sometimes, we may write  $u(z)$ instead of $\Phi(u)(z)$ for $u \in G$ and $z \in T$, by a slight abuse of notation. 
\end{rem}

Recall the presentation \ref{present-he3} of $\Heis(3)$ given in Theorem \ref{theorem-rigid-hyperelliptic}.

\begin{lemma} \label{well-defined}
	Let $\Phi \colon G \hookrightarrow \Aut(T)$ be a  faithful rigid holomorphic action in standard form. Then:
	\begin{enumerate}
	\item If $G = \mathbb Z_3^2$, the following three elements are zero in $T$:	
	\begin{itemize}
	\item[]  $v_1 := \left((\zeta_3-1)c_1, \ (1-\zeta_3)b_2, \ (1-\zeta_3)b_3, \ (1-\zeta_3)b_4 \right)$,
	\item[] $v_2 := \left(0, \ 3b_2, \ 0, \ 0\right)$,
	\item[] $v_3 := \left(3c_1, \ 0, \ 0,\ 0\right)$.
	\end{itemize}
Conversely, given $b_j$ and $c_1$ such that $v_1, v_2, v_3$  are zero in $T$, we obtain a 
	faithful holomorphic action of $\mathbb Z_3^2$ on $T$ in standard form.
	\item If $G = \Heis(3)$, the elements $v_1, \ldots, v_6$ are zero in $T$, where
	\begin{itemize}
	\item[] $v_4 := \left((1-\zeta_3)a_1-c_1, \ b_4-\zeta_3b_2 + (1-\zeta_3)a_2, \ b_2-\zeta_3b_3, \ b_3-\zeta_3b_4 + (1-\zeta_3^2)a_4\right)$, 
	\item[] $v_5 := \left(3a_1, \ a_2+a_3+a_4, \ a_2+a_3+a_4, \ a_2+a_3+a_4\right)$, 
	\item[] $v_6 := \left(0, \ (1-\zeta_3)a_2, \ (1-\zeta_3)a_3, \ (1-\zeta_3)a_4\right)$
	\end{itemize}
	and $v_1, v_2, v_3$ are as above.\\
		Conversely, given $a_i$, $b_j$ and $c_1$ such that $v_1, \ldots , v_6$  are zero in $T$, we obtain a 
	faithful holomorphic action of $\Heis(3)$ on $T$ in standard form. 
\end{enumerate}
\end{lemma}

\begin{proof}
	We will only sketch the proof of (2). Since $\Phi$ is a homomorphism, the images $\Phi(g)$, $\Phi(h)$ and $\Phi(k)$ of the generators fulfill the six defining relations of 
	$\Heis(3)$.  As an example, the relation $\Phi(g) \circ \Phi(k) = \Phi(k) \circ \Phi(g)$ precisely means that the difference $v_6$ of 
	\begin{align*}
	&gk(z) = (z_1 + a_1 + c_1, \ \zeta_3 z_4 + a_2, \ \zeta_3 z_2 + a_3, \ \zeta_3 z_3 + a_4), ~  \text{ and}\\
	&kg(z) = (z_1 + a_1 + c_1, \ \zeta_3 z_4 + \zeta_3 a_2, \ \zeta_3 z_2 + \zeta_3 a_3, \ \zeta_3 z_3 + \zeta_3a_4)
	\end{align*}
	is zero, where $z = (z_1,\ldots,z_4)$.  The remaining five relations yield the vanishing of the other $v_i's$.  \\
	Conversely, if the elements $v_1, \ldots , v_6$ are zero, then it is clear that the formulae in Remark \ref{heis-action}
	define a faithful homomorphism $\Phi$. 
\end{proof}

\begin{cor} \label{cor-c1-b2-ord3}
	Let $\Phi \colon G \hookrightarrow \Aut(T)$ be a rigid and free action in standard form, then:
	\begin{itemize}
		\item[(a)] The elements $c_1$ and $b_2, b_3, b_4$ have order $3$ in $E_1$ and $E_2$, $E_3$, $E_4$, respectively.
		\item[(b)] Let $p_i \colon K \to E_i$ be the projection on the $i$-th coordinate. Then $p_i(K)$ does not contain 
		\[
		E_i[3]:=\lbrace z_i \in E_i ~ \vert ~ 3z_i=0 \rbrace. 
		\]
	\end{itemize}
\end{cor}

\begin{proof}
	(a) We use the important property that the maps 
	\begin{align*}
	E_i \hookrightarrow  E_1 \times \ldots \times E_4 \to T =  (E_1 \times \ldots \times E_4)/K
	\end{align*}
	are injective: it implies that if three of the four coordinates of an element of $K$ are zero, then the fourth coordinate is zero as well. It then follows from $v_2, v_3 \in K$ that $3c_1 = 0$, $3b_2 = 0$. Moreover, $c_1$ and $b_2$ are clearly non-trivial, since else $h$ resp. $k$ do not act freely. In order to show that $\ord(b_3) = \ord(b_4) = 3$, a similar argument works: we observe that 
	\begin{align*}
	&hk(z) = (\zeta_3 z_1 + \zeta_3c_1, \ \zeta_3 z_2 + b_2, \ z_3 + b_3, \ \zeta_3^2 z_4 + b_4), \\
	&hk^2(z) = (\zeta_3 z_1 + 2\zeta_3 c_1, \ \zeta_3^2 z_2 + b_2, \ \zeta_3z_3 + b_3, \ z_4 + b_4)
	\end{align*}
	have order $3$, and so, the elements $(0, \ 0, \ 3b_3, \ 0), (0, \ 0, \ 0, \ 3b_4)$ are zero in $T$, which in turn implies that $3b_3 = 3b_4 = 0$ in $E$. Moreover, $b_3$ and $b_4$ cannot be zero, because the elements  $hk$ and $hk^2$ act freely by assumption.  \\
	(b) If $p_i(K)$ contains $E_i[3]$, we can find an element in $K$ whose first coordinate is $c_1$ (if $i = 1$) or $b_i$ (if $i \in \{2,3,4\}$). Thus, we obtain a fixed point of $k$ (if $i = 1$), $h$ (if $i = 2$), $hk$ (if $i = 3$) or $hk^2$ (if $i = 4$). 
\end{proof}

\begin{assump} \label{free-assump}
	Since we are only interested in free actions, we can and will  assume from now on that $\ord(c_1) = 3$ and $\ord(b_j) = 3$   for all $j$. 
\end{assump}

We will denote the fixed locus of $\zeta_3 \in \Aut(E)$ by 
\begin{align*}
\Fix_{E}(\zeta_3) := \{x \in E \ |\ \zeta_3 x = x\} = \lbrace 0, t,-t \rbrace, \quad \makebox{where} \quad t := \frac{1+2\zeta_3}{3}.
\end{align*} 


Besides knowing  the fixed locus of $\zeta_3$, we also need a description of  the kernel of $3(\zeta_3-1)$. 

\begin{lemma} \label{9-torsion-span}
	Let $E = \CC/\ZZ[\zeta_3]$ be the Fermat elliptic curve.
	\begin{itemize}
		\item[(a)] The kernel of the multiplication map $3(\zeta_3-1) \colon E \to E$ is isomorphic to $\ZZ_3 \times \ZZ_9$. More precisely, it is spanned by $1/3$ and $t/3 = (1+2\zeta_3)/9$. 
		\item[(b)] Let $x \in E$ be an element of order $9$. If $3x$ is fixed by $\zeta_3$, then $x$ and $\zeta_3 x$ span $\ker(3(\zeta_3-1))$.
	\end{itemize}
\end{lemma}

\begin{proof}
	(a) This is just a calculation. \\
	(b) By (a), it suffices to exclude that $\langle x, \zeta_3 x \rangle \simeq \ZZ_9$. Assume for a contradiction that $\zeta_3 x \in \langle x \rangle$. Since $x$ has order $9$, $\zeta_3$ is an automorphism of $\langle x \rangle$ of order $3$ so that we are left with the possibilities $\zeta_3 x \in \{4x,7x\}$. In these cases, $(\zeta_3-1)x = \pm 3x$. Multiplying both sides 
	with  $(\zeta_3-1)$ yields 
	\begin{align*}
	-3\zeta_3x=(\zeta_3-1)^2 x =  \pm(\zeta_3-1)3x = 0,  
	\end{align*}
	which  implies $3x=0$,  in contradiction to 
	the assumption $\ord(x)=9$. 
\end{proof}

\begin{prop} \label{prop-structure-K}
	Let $\Phi \colon G \hookrightarrow \Aut(T)$ be a rigid and free action in standard form, where $T = E^4/K$. Then, $K \leq \Fix_E(\zeta_3)^4 \simeq \mathbb Z_3^4$. In particular, every non-trivial element of $K$ has order $3$. 
\end{prop}

\begin{proof}
	We preliminarily remark that for all $u\in G$ we may view $\rho(u)$ as an automorphism of $E^4$ that maps $K$ to $K$.
	Let $(t_1,t_2,t_3,t_4) \in K$. By the initial remark, the element
	\begin{align*}
	(\rho(h)-\id_{E^4})(\rho(hk)-\id_{E^4})(\rho(hk^2)-\id_{E^4})(t_1,t_2,t_3,t_4) = ((\zeta_3-1)^3 t_1, \ 0, \ 0, \ 0).
	\end{align*}
	is contained in $K$ as well. Since $(\zeta_3-1)^2 = -3\zeta_3$ and $E_1 \hookrightarrow T$ is injective, we infer that $3t_1 \in \Fix_{E}(\zeta_3)$. Analogous arguments yield that $3t_i \in \Fix_{E}(\zeta_3) \simeq \mathbb Z_3$.
	It follows in particular that $t_i \in E_i[9]$ for all $i$.
	Suppose that there is $i$ such that $\ord(t_i) = 9$. Then, we arrive at a contradiction as follows: by Lemma \ref{9-torsion-span} (b), the two elements $t_i$, $\zeta_3 t_i$ span a subgroup of $E_i[9]$  isomorphic to $\ZZ_3 \times \ZZ_9$. Thus, the image of the projection 
	$p_i\colon K \to E_i$ contains $E_i[3]$, which is impossible 
	by Corollary \ref{cor-c1-b2-ord3} (b). 
	This shows that $3t_1 = 3t_2 = 3t_3 = 3t_4 = 0$. If some $t_i$ is  not fixed by $\zeta_3$, then $t_i$ and $\zeta_3t_i$ span 
	$E_i[3]$, and we again arrive at  the contradiction  $E_i[3]\leq p_i(K)$.
\end{proof}

\begin{rem} \label{remark-K}
Let $X=T/G$ be a rigid hyperelliptic fourfold, where $T=E^4/K$. By construction, $K$ is the 
kernel of the addition map $E^4 \to T$. 
By Proposition \ref{prop-structure-K}, $K \leq \Fix_E(\zeta_3) \simeq \mathbb Z_3^4$. Moreover, $K$ does not contain any
non-zero multiples of unit vectors. 
There are $129$ subgroups  of $\mathbb Z_3^4$ with these two properties, the set of which we denote by 
$\mathcal K$.
\end{rem}

\subsection{Kernels and Actions for $\Heis(3)$}
In this subsection we finish part (1) and (2) of the classification scheme for $G=\Heis(3)$. 
Here, only those elements in $\mathcal K$ stabilized by  $\rho(g)$ can occur as a kernel. 
Out of the $129$ possibilities found in 
Remark \ref{remark-K}, just the following nine survive:

\medskip
\begin{center}
	\bgroup\def\arraystretch{1.3}\begin{tabular}{|c|l|} \hline 
	$\dim_{\ZZ_3}(K)$ & $K$ \\ \hline \hline
	$0$ & $\{0\}$ \\  \hline
	$1$ & $\langle(0,t,t,t)\rangle$, $\langle(t,t,t,t)\rangle$, $\langle (-t,t,t,t)\rangle$ \\ \hline
	$2$ & $\langle (0,t,t,t), (0,t,-t,0)\rangle$, \\
		& $\langle (0,t,t,t), (t,-t,0,t)\rangle$, \\ 
		& $\langle(0,t,t,t),(t,t,0,-t)\rangle$ \\ \hline
	$3$ & $\langle (0,t,t,t),(0,t,-t,0), (t,0,t,0)\rangle$, \\
		& $\langle (0,t,t,t), (t,-t,-t,t), (t,t,0,t)\rangle$ \\ \hline
	\end{tabular}\egroup
\end{center}

\medskip
As a first step, we exclude the last four possibilities for $K$ in  the table.

\begin{prop} \label{kernel-prop}
Let $T/\Heis(3)$ be a rigid hyperelliptic fourfold, where $T = E^4/K$. Then, $K$ is either generated by an element $(t',t'',t'',t'')$ where $t',t'' \in \Fix_E(\zeta_3)$, or 
$K = \langle (0,t,t,t), (0,t,-t,0)\rangle$.
\end{prop}

\begin{proof}
The cyclic case follows from Remark \ref{remark-K}. If $K$ is non-cyclic, investigating the five subgroups introduced in Remark \ref{remark-K} shows that $(0,t,t,t) \in K$ in each case. From this and since the last three coordinates of
\begin{align*}
 v_5 = \left(3a_1, \ a_2+a_3+a_4, \ a_2+a_3+a_4, \ a_2+a_3+a_4\right) \in K \qquad \makebox{($v_5$ as in Lemma \ref{well-defined})}
\end{align*}
coincide, we conclude that $3a_1 = 0$. This shows that $(1-\zeta_3)a_1$ is fixed by $\zeta_3$. Since $(1-\zeta_3)a_1 - c_1$ is the first coordinate of $v_4 \in K$, it is fixed by $\zeta_3$ as well. We obtain that $c_1 \in \Fix_E(\zeta_3)$, and hence, $K$ cannot contain an element whose first coordinate is non-zero: otherwise, $k$ does not act freely. Out of the five non-cyclic possibilities, only $K = \langle (0,t,t,t), (0,t,-t,0)\rangle$ remains. 
\end{proof}

We observe in particular that the above proposition shows that every element $(t_1,t_2,t_3,t_4) \in K$ has the property that $t_2+t_3+t_4 = 0$ in $E$. This observation is useful to prove a simple criterion for the freeness of the action of $\Heis(3)$ on $T = E^4/K$.

\begin{lemma} \label{lemma-freeness}
	An action $\Phi \colon \Heis(3) \hookrightarrow \Aut(T)$  is free if and only if none of the elements
	$k, g,  h, gh$ and $gh^2$ have a fixed point. If $\Phi$ is in standard form, the 
	freeness of these elements can be characterized as follows:
	\begin{center}
	\begin{tabular}{ccl}
		(1) & $k$: &
		$c_1$ is never the first coordinate of an element in $K$, \\
		(2) & $g$: & 
		$(\zeta_3-1)a_1 \neq 0$ in $E$,\\
		(3)&  $h$: &
		$b_2$ is never the second coordinate of an element in $K$,\\
		(4) & $gh$: &
		$\zeta_3^2(a_2+a_3) + a_4 + \zeta_3^2(b_2+b_4) + b_3 \neq 0$ in $E$, \\
		(5) & $gh^2$: &
		$\zeta_3(a_2+a_3)+a_4 - \zeta_3(b_2+b_3) - b_4 \neq 0$ in $E$. 
	\end{tabular}
\end{center}
\end{lemma}

\begin{rem}
	Later we will show that $K=\langle (0, \ t, \ t, \ t) \rangle$ or
	$\langle (0, \ t, \ -t, \ 0),  (0, \ t, \ t, \ t) \rangle$. In particular condition (1) will be obsolete, since $c_1\neq 0$ in $E_1$ (Assumption \ref{free-assump}). Furthermore, condition (3) then  translates into  the condition that $(\zeta_3-1)b_2 \neq 0$.
\end{rem}

\begin{proof}[Proof of Lemma \ref{lemma-freeness}]
	Note that an element $u \in \Heis(3)$ acts freely if and only if $u^2$ acts freely, because $u^3=1$ for all $u \in \Heis(3)$.
	Observe that $(gh)^2$ is conjugate to $g^2h^2$ and $(g^2h)^2$ is conjugate to $gh^2$.
	Thus, the representatives  $g$, $g^2$,$h$, $h^2$, $k$, $k^2$, $gh$, $g^2h$, $gh^2$ and $g^2h^2$ of the ten non-trivial  conjugacy classes act freely if and only if
	$g$, $h$, $k$, $gh$ and $gh^2$ act freely. \\
	Clearly, $k$ and $h$ act freely if and only if (1) and (3) hold, respectively. The element $g$ has a fixed point on $T$ if and only if there is $z = (z_1, \ldots, z_4) \in T$ such that
	\begin{align*}
	&g(z) = (z_1 + a_1, \ z_4 + a_2, \ z_2 + a_3, \ z_3 + a_4) = (z_1, \ldots, z_4) \text{ in } T \\
	\iff &(a_1, \ z_4 - z_2 + a_2, \ z_2 - z_3 + a_3, \ z_3 - z_4 + a_4) = 0 \text{ in } T.
	\end{align*}
	Reading the second line of the above equivalence as an equation in $E^4$ and setting $w_2 := z_4 - z_2$, $w_3 := z_2 - z_3$, we infer that the above is satisfied if and only if there is $(t_1, t_2, t_3, t_4) \in  K$ such that
	\begin{equation} \label{g-fp-eq}
	(a_1, \ w_2 + a_2, \ w_3 + a_3, \ -w_2 - w_3 + a_4) = (t_1,t_2,t_3,t_4) \text{ in } E^4. 
	\end{equation}
	Thus, equation \ref{g-fp-eq} can only be satisfied if $a_1 = t_1$, $w_2 = t_2 - a_2$, $w_3 = t_3 - a_3$. Calculating the fourth coordinate, we obtain that
	\begin{align*}
	-w_2 - w_3 + a_4 = -t_2 - t_3 + a_2 + a_3 + a_4 = t_4.
	\end{align*}
	Since $t_2 + t_3 + t_4 = 0$ (Proposition \ref{kernel-prop}), we obtain that $a_2 + a_3 + a_4 = 0$. \\
	In total, this shows that $g$ has a fixed point on $T$ if and only if $(\zeta_3-1)a_1 = 0$ and $a_2 + a_3 + a_4 = 0$. It follows that $(\zeta_3-1)a_1 \neq 0$ is a sufficient condition for the freeness of $g$. Moreover, this condition is necessary for the freeness of $k$, since otherwise the first coordinate of the vector $-v_4 = 0 \in T$ defined in Lemma \ref{well-defined} is $c_1$, which is impossible by (1). \\
	Conditions (4) and (5) mean that $gh$ and $gh^2$ act freely, respectively: this is proved as above.
\end{proof}

We are now in the situation to complete step (1) of our classification scheme (see p. \pageref{scheme-first}) in the Heisenberg case. In fact, the upcoming proposition shows that only the two kernels $K_1$ and $K_2$ from Theorem \ref{Mani} \ref{Mani2} allow a free $\Heis(3)$-action.

\begin{prop}\label{propExcludeK}
Let $\phi \colon \Heis(3) \hookrightarrow \Aut(E^4/K)$ be a  free action in standard form. Then $K \neq \{0\}$ and $K \neq \langle (t',t,t,t)\rangle$ for $t' \neq 0$.

\end{prop}

\begin{proof}
Let $v := \zeta_3^2(a_2+a_3) + a_4 + \zeta_3^2(b_2+b_4) + b_3$ be the element of  Lemma \ref{lemma-freeness} (4). A direct calculation yields
	\begin{align*}
	(gh)^3(z) = (z_1, \ z_2 + \ze_3 v, \ z_3 + \zeta_3 v, \ z_4 + v).
	\end{align*}
Since $(gh)^3$ is the identity on $E^4/K$, the element
	\begin{align*}
		(0, \ \zeta_3 v, \ \zeta_3 v, \ v)
	\end{align*}
	is contained in the kernel $K$. Now, if $v \neq 0$, then $K \neq \langle (t',t,t,t)\rangle$ for $t' \neq 0$. On the other hand, if $K = \{0\}$, then $v = 0$ and hence Lemma \ref{lemma-freeness} (4) shows that the action of $\Heis(3)$ is not free.
\end{proof}

Let $T = \CC^4/\Lambda$, where $\Lambda$ is one of the lattices
\begin{align*}
\Lambda_1 := \ZZ[\zeta_3]^4 + \langle (0, \ t, \ t, \ t)\rangle \ \ \ \text{ or } \ \ \ \Lambda_2 := \Lambda_1 + \langle (0, \ t, \ -t, \ 0)\rangle
\end{align*}
corresponding to the two remaining kernels $K_1$ and $K_2$. The results from above allow us to determine all free actions $\Phi \colon \Heis(3) \hookrightarrow \Aut(T)$ in standard form. For a more efficient computation, we prove a refinement of 
Corollary \ref{cor-c1-b2-ord3}:

\begin{lemma} \label{elementary-properties}
	Let $T$ be one of the two complex tori above, and let $\Phi \colon \Heis(3) \hookrightarrow \Aut(T)$ be a free action in standard form. Then, the following statements hold in $E$:
	\begin{itemize}
		\item[(a)] $(1-\zeta_3)a_1 = c_1$ and $a_1$ is an element of order $3$ which is not fixed by $\zeta_3$.
		\item[(b)] The elements $b_2, b_3$ and $b_4$ of order $3$ are not fixed by $\zeta_3$. In particular, the element  $v_1$ defined in  
		Lemma \ref{well-defined} is equal to 
		$\pm (0,t,t,t)$ in $E^4$.
		\item[(c)] There are $j_3,j_4$ such that $b_3 = \zeta_3^{j_3} b_2$ and $b_4 = \zeta_3^{j_4} b_2$.
		\item[(d)] $\ord(a_2), \ord(a_3), \ord(a_4) \in \{1,3\}$.
	\end{itemize}
\end{lemma}

\begin{proof}
	(a) The first coordinate of every element of $K$ is zero, and thus, the property that $v_4$ and $v_5$ are zero in $T$ shows that $3a_1 = 0$ and $(1-\zeta_3)a_1 = c_1$. Now, if $a_1$ is fixed by $\zeta_3$, the first coordinate of $v_4$ is $-c_1$. Thus, $c_1 = 0$ and $k$ does not act freely, a contradiction. \\
	(b) If $b_2$ (resp. $b_3$, $b_4$) is fixed by $\zeta_3$, then the torsion group $K$ contains an element whose second (resp. third, fourth) coordinate is $b_2$ (resp. $b_3$, $b_4$): in this case, the element $h$ (resp. $hk$, $hk^2$) does not act freely on $T$. The last statement follows from the description of the possible kernels $K$ since $(1-\zeta_3)b_i \neq 0$. \\
	(c) It was shown in (b) that $b_2$, $b_3$ and $b_4$ are elements of order $3$ that are not fixed by $\zeta_3$. Hence
		\begin{align} \label{b3b4-not-fixed}
		b_3,b_4 \in E[3] \setminus \Fix_{E}(\zeta_3) = \{b_2,\ \zeta_3 b_2,\ \zeta_3^2 b_2,\ -b_2,\ -\zeta_3 b_2,\ -\zeta_3^2 b_2\}.
		\end{align} 
		Furthermore, the second assertion of part (b) yields
		\begin{align*}
		(1-\zeta_3)b_2 = (1-\zeta_3)b_3 = (1-\zeta_3)b_4 \in \{t,-t\}.
		\end{align*}
		In particular, $b_2-b_3$ and $b_2-b_4$ are contained in $\Fix_{E}(\zeta_3)$. Combined with \ref{b3b4-not-fixed}, this easily implies that $b_3, b_4 \in  \{b_2, \zeta_3 b_2, \zeta_3^2 b_2\}$. \\
	(d) Since $v_6$ is contained in $K$, the elements $(1-\zeta_3)a_i$ are fixed by $\zeta_3$ for $i \in \{2,3,4\}$. This is the case if and only if $3a_i = 0$. 
\end{proof}

\begin{rem}\label{13rem}
	Let $\Phi \colon \Heis(3) \hookrightarrow \Aut(T)$ be an action in standard form. In the case where $b_2$ is not fixed by $\zeta_3$, it holds  
	$b_2 = (-1)^i \zeta_3^j/3$, and conjugation by $(-1)^i \zeta_3^{-j} \id_T$ yields an action in standard form with $b_2=1/3$. 
\end{rem}

\begin{rem}\label{ActHeis3}
According to part (b) and Remark \ref{13rem}, we may assume that $b_2 = 1/3$. By part (a), we have $a_1 = \pm \zeta_3^j/3$, and, up to conjugation of the action in the first coordinate, $a_1 = 1/3$ (cf. Remark \ref{13rem}).  Thus $c_1 = (1-\zeta_3)/3$. \\
It is now easy to construct all free actions $\Phi$:
\begin{itemize}
	\item run over all $j_3, j_4 \in \{0,1,2\}$ and define $b_3 := \zeta_3^{j_3}/3$, $b_4 := \zeta_3^{j_4}/3$,
	\item run over all $a_2,a_3,a_4 \in E[3]$ and define the elements $v_1, \ldots, v_6$ of Lemma \ref{well-defined},
	\item check if $v_1, \ldots, v_6$ are contained in $\Lambda_i$ and if the freeness conditions of Lemma \ref{lemma-freeness} are satisfied.
\end{itemize}
Running a MAGMA implementation, we find $108$ actions for the  lattice $\Lambda_1$ and $324$ for the lattice $\Lambda_2$, some of which may coincide on $\mathbb C^4/\Lambda_i$. 
For both lattices $\Lambda_i$, there are four special cohomology classes in $H^1(\Heis(3),\mathbb C^4/\Lambda_i)$.
\end{rem}

\subsection{Kernels and Actions for $\mathbb Z_3^2$}

In the $\mathbb Z_3^2$ case, kernels $K \in \mathcal K$ are clearly invariant under the action of $\mathbb Z_3^2$ via $\rho$. However, different kernels might yield isomorphic hyperelliptic manifolds. In order to take this into account, we define a group action on $\mathcal K$ such that kernels belonging to different orbits cannot correspond to isomorphic hyperelliptic manifolds. 

\begin{prop} \label{Z3^2-normalizer}
	Any biholomorphism $f \colon X \to X'$ between two hyperelliptic manifolds with holonomy group $\ZZ_3^2$ is induced by a biholomorphic map $\hat{f} \colon E^4 \to E^4$, $z \mapsto Az+d$. This means that  $A$ is contained in the normalizer $N_{\Aut_0(E^4)}(\ZZ_3^2)$.   
\end{prop}

\begin{proof}
	Write $X = T/\ZZ_3^2$ and $X'=T'/\ZZ_3^2$, where $T = E^4/K$ and $T' = E^4/K'$. 	
	 By Corollary \ref{GeometricBieber} (4), the isomorphism $f \colon X  \to X'$ is induced by an affinity  $\alpha(z) = Az + d$ with $A\in \GL(4,\mathbb C)$. Moreover there exists an automorphism $\varphi \in \Aut(\ZZ_3^2)$ such that  $A \rho A^{-1} = \rho \circ \varphi$.
	Note that the only matrix in the image of $\rho$ having the eigenvalue $\zeta_3$ with multiplicity $3$ is $\rho(k)$, thus 
	we infer that $\varphi(k) = k$. On the other hand, $\varphi(h) \in \{h, hk, hk^2\}$ because the matrices $\rho(h)$, $\rho(hk)$ and 
	$\rho(hk^2)$ have the same eigenvalues and all other matrices do not.
	It follows that $A$ takes one of the following forms:
	\begin{align*}
	\begin{pmatrix}
	a_1 &  &  &  \\ 
	 & a_2 &  &  \\
	 &  & a_3 &  \\
	 &  &  & a_4
	\end{pmatrix}, \quad \begin{pmatrix}
	a_1 & 0 & 0 & 0 \\
	0 & 0 & a_2 & 0 \\
	0 & 0 & 0 & a_3 \\
	0 & a_4 & 0 & 0
	\end{pmatrix} \quad \text{ or } \quad \begin{pmatrix}
	a_1 & 0 & 0 & 0 \\
	0 & 0 & 0 & a_2 \\
	0 & a_3 & 0 & 0 \\
	0 & 0 & a_4 & 0
	\end{pmatrix}.
	\end{align*}
	Note that for the above matrices, $A\cdot e_i$ is a vector with only one non-zero entry. Denoting by $\Lambda$ and $\Lambda'$ the lattices of $T$ and $T'$, respectively, this observation, together with the condition $A \cdot \Lambda = \Lambda'$  shows  that the entries of $A$ are Eisenstein integers. The same holds  for $A^{-1}$, and it follows that  the complex numbers $a_i$ are contained in $\ZZ[\zeta_3]^* = \langle - \zeta_3 \rangle$, hence $A \cdot \ZZ[\zeta_3]^4 = \ZZ[\zeta_3]^4$. 
\end{proof}
As a by-product of the proof we obtain: 
\begin{cor}\label{NormZ3hol}
	The normalizer  $N_{\Aut_0(E^4)}(\ZZ_3^2)$ is a finite group with 
	$6^4\cdot3=3888$ elements generated by the matrices
	\[
	\begin{pmatrix}
	-\zeta_3 & & & \\
	& 1 & & \\
	& & 1 & \\
	& & & 1
	\end{pmatrix}, \quad \begin{pmatrix}
	1 & & & \\
	& -\zeta_3 & & \\
	& & 1 & \\
	& & & 1
	\end{pmatrix}, \quad \begin{pmatrix}
	1 & 0 & 0 & 0 \\
	0 & 0 & 1 & 0 \\
	0 & 0 & 0 & 1 \\
	0 & 1 & 0 & 0
	\end{pmatrix}.
	\]
	As an abstract group, $N_{\Aut_0(E^4)}(\ZZ_3^2)$ is isomorphic to $\mathbb Z_6^4 \rtimes \mathcal A_3$. 
\end{cor}

Similarly as in the Heisenberg case, we formulate a simple criterion for the freeness of an action $\Phi \colon \mathbb Z_3^2 \hookrightarrow \Aut(T)$ in standard form. 

\begin{lemma}
An action $\Phi \colon \mathbb Z_3^2 \hookrightarrow \Aut(T)$  is free if and only if none of the elements
$h, k, hk$ and $hk^2$ have a fixed point. If $\Phi$ is in standard form, the 
latter holds precisely when
\begin{itemize}
	\item[(1)]
	$c_1$ is never the first coordinate of an element in $K$,
	\item[(2)]
	$b_j$ is never the $j$-th coordinate of an element in $K$ for all $j$.
\end{itemize}
\end{lemma}

Observe that Corollary \ref{cor-c1-b2-ord3} shows that the elements $c_1$, $b_2$, $b_3$ and $b_4$ have order $3$, respectively. In particular, there are only finitely many free actions in standard form. Our MAGMA code determines all of these actions. The  number of these actions is displayed in  the middle column of the table from below. \\

\begin{rem}\label{13Shaolin}
 Note that  the normalizer  $N_{\Aut_0(E^4)}(\ZZ_3^2)$  acts on $\mathcal K$. Proposition  \ref{Z3^2-normalizer} tells us that two hyperelliptic fourfolds  $X$  and $X'$ cannot be biholomorphic if the corresponding kernels $K,K' \in \mathcal K$ belong to  different orbits of this action. Performing a MAGMA computation,  we find exactly $13$ orbits:  
	
	\medskip
	\begin{center}
	\bgroup\def\arraystretch{1.3}\begin{tabular}{|c|l|l| l |} \hline
			$i$  & representative $K_i$ of the orbit  &   \# of free actions & \# of special classes in $H^1(\mathbb Z_3^2, E^4/K_i)$  \\ 
			\hline \hline
			$1$ & $\{0\}$ & $16$ & $16$\\ 
			$2$ & $\langle (0,0,t,t)\rangle$ & $72$ & $8$ \\
			$3$ & $\langle (0,t,t,t)\rangle$ & $108$ & $12$ \\
			$4$ & $\langle (t,0,0,t) \rangle$ & $72$ & $8$ \\
			$5$ & $\langle (t,0,t,t)\rangle$ & $108$ & $12$ \\
			$6$ & $\langle (t,t,t,t) \rangle$ & $162$ & $18$  \\ \hline
			$7$ & $\langle (0,t,t,t), (0,t,-t,0)\rangle$ & $108$ & $4$ \\
			$8$& $\langle(0,0,t,t),\ (t,0,-t,0)\rangle$ & $108$ & $4$  \\
			$9$& $\langle (t,t,0,0),\ (0,0,t,t)\rangle$ & $324$ & $4$ \\
			$10$& $\langle (0,0,t,t),\ (t,t,0,t) \rangle$ & $162$ & $2$ \\
			$11$& $\langle (t,0,0,t),\ (t,t,t,-t) \rangle$ & $162$  & $2$ \\
			$\exc$ & $\langle (t,-t,0,t),\ (t,t,-t,0) \rangle$ & $0$ & $0$ \\ \hline
			$12$ & $\langle (-t,t,0,0),\ (t,0,t,t),\ (t,t,t,0) \rangle$ & $486$  & $6$ \\ \hline
		\end{tabular}\egroup
	\end{center}
\medskip
For all kernels $K_i$, except for $K_{\exc}$, there are rigid and free actions on $T_i=E^4/K_i$. Thus there are at least $12$ biholomorphism classes of rigid fourfolds with holonomy $\mathbb Z_3^2$. Moreover, every hyperelliptic fourfold with holonomy $\ZZ_3^2$ is obtained as a quotient of $T=E^4/K_i$.

\end{rem}

\section{Biholomorphism and Diffeomorphism Classes of Rigid Hyperelliptic Fourfolds} \label{section-Bihol}

In this section, we perform the final step of our classification scheme, outlined at the end of Section  \ref{bieberer-section}. 
The main part is devoted to determine for each group $G=\Heis(3)$ and  $\mathbb Z_3^2$ and for each kernel $K$ the normalizers
$\mathcal N_{\mathbb C}(\Lambda_K)$ and $\mathcal N_{\mathbb R}(\Lambda_K)$ explicitly, where $\Lambda_K=\mathbb Z[\zeta_3]^4 + K$. With the help of MAGMA,  we then determine  the orbits of the action of these groups on the special cohomology classes in $H^1(G,E^4/K)$.  Propositions \ref{TParts-Heis3} and \ref{TParts-Z3^2} impose conditions on the coboundaries and therefore allow an efficient computer search.
 Thereby, we distinguish all isomorphism and diffeomorphism  classes of rigid hyperelliptic fourfolds and prove our main theorems. In the code, the 
special cohomology classes  are represented by the translation parts of rigid and free $G$-actions in standard form.


\subsection{The Case $\Heis(3)$}


To compute the sets $\mathcal N_{\mathbb R}(\Lambda,\Lambda')$ and $\mathcal N_{\mathbb C}(\Lambda,\Lambda')$, where $\Lambda$ and $\Lambda'$ are one of 
\begin{align*}
	\Lambda_1 = \mathbb Z[\zeta_3]^4 + \langle (0,t,t,t)\rangle, \qquad \Lambda_2 = \Lambda_1 + \langle (0,t,-t,0)\rangle,
\end{align*} 
respectively, we will use the real representation theory of $\Heis(3)$, which we 
briefly recall.
Throughout the subsection, we will use the notation $\Lambda_i^\natural := \Lambda_i \cap (\lbrace 0 \rbrace \times \mathbb C^3)$ and exploit the fact that  $$\Lambda_i= \mathbb Z[\zeta_3] \oplus \Lambda_i^\natural.$$
\begin{rem}
The Heisenberg group $\Heis(3)$ has $5$ irreducible non-trivial real representations:  
\begin{itemize}
\item
Four  $2$-dimensional representations induced from the central quotient  $\mathbb Z_3^2$ by inflation. They map $(a,b) \in \mathbb Z_3^2$ to $B^b, \ B^{a+b}, \ B^a \quad \makebox{and} \quad B^{2a+b}$, respectively, where $B:=
-\frac{1}{2}
\begin{pmatrix} 1 &  \sqrt{3}  \\ -\sqrt{3} &1 \end{pmatrix}$.
\item
One irreducible $6$-dimensional representation, which is  the decomplexification $\rho_{3\mathbb R}$ of the representation $\rho_3$ defined in \ref{rho3}.
\end{itemize} 
The endomorphism algebra of each of these representations is isomorphic to  $\mathbb C$. 
\end{rem}

\begin{rem}\label{rem1}
A matrix $A \in \GL(8,\mathbb R)$ is contained in  $\mathcal N_{\mathbb R}(\Lambda,\Lambda')$ if and only if 
\begin{align}\label{ExConj}
A\rho_{\mathbb R} A^{-1} =\rho_{\mathbb R} \circ \varphi \qquad \makebox{for some} \qquad \varphi \in  \Aut(\Heis(3))
\end{align} 
and $A\cdot \Lambda =\Lambda'$. 
Since $\rho_{\mathbb R}=\rho_{1\mathbb R} \oplus \rho_{3\mathbb R}$
is the sum of an irreducible real  $2$-dimensional and an irreducible $6$-dimensional representation, 
 Schur's Lemma implies that $A$ is of block diagonal form 
\[
A=\diag(C,D), \quad \makebox{where} \quad  C \in \GL(2,\mathbb R)  \quad \makebox{and} \quad  D \in \GL(6,\mathbb R). 
\]
In particular, the  characters $\chi_{i \mathbb R}$  of $\rho_{i \mathbb R}$ are stabilized by  $\varphi$, i.e., $\chi_{i \mathbb R} = \chi_{i \mathbb R} \circ \varphi$. We write  
$\varphi \in \Stab(\chi_{1 \mathbb R}, \chi_{3 \mathbb R})$. 
Note that $\Stab(\chi_{1 \mathbb R}, \chi_{3 \mathbb R}) = \Stab(\chi_{1 \mathbb R})$  because  $\rho_{3\mathbb R}$ is the unique irreducible $6$-dimensional real 
representation of $\Heis(3)$.\\ On the other hand, given $\varphi \in \Stab(\chi_{1 \mathbb R})$, 
there exist invertible matrices  $C_{\varphi}\in \GL(2,\mathbb R)$ and $D_{\varphi} \in \GL(6,\mathbb R)$ such that 
$A_{\varphi}:=\diag(C_{\varphi},D_{\varphi})$ fulfills the equation \ref{ExConj}. \\
Clearly, the matrices  $C_{\varphi}$ and $D_{\varphi}$  are unique only up to a non-zero element in the endomorphism algebras
$\End_{\Heis(3)}(\rho_{1\mathbb R})$ and $\End_{\Heis(3)}(\rho_{3\mathbb R})$, respectively.  Both of these algebras are  isomorphic to $\mathbb C$. 
Any $\varphi \in \Stab(\chi_{i\mathbb R})$ either stabilizes $\chi_i$ or maps $\chi_i$ to $\overline{\chi}_i$ because 
 $\chi_{i \mathbb R}$ is equal to the sum of the complex characters $\chi_i$ and $\overline{\chi}_i$.  Moreover, a matrix defining an 
 equivalence between $\rho_{i \mathbb R}$  and $\rho_{i \mathbb R}\circ \varphi$ is $\mathbb C$-linear precisely if $\chi_i$ is stabilized and $\mathbb C$-antilinear if and only if
$\chi_i$ is mapped to $\overline{\chi}_i$, cf. Remark \ref{FundBie}.
In particular, $A_{\varphi}=\diag(C_{\varphi},D_{\varphi})$ is $\mathbb C$-linear if and only if $\varphi \in \Stab(\chi_1,\chi_3)$. 
\end{rem}

\begin{rem}\label{Antilinear}
Suppose that $C \in \GL(2,\mathbb R)$ and $\varphi \in \Stab(\chi_{1\mathbb R})$ such that 
$C\rho_{1\mathbb R}(u)C^{-1} =\rho_{1 \mathbb R}(\varphi(u))$. 
If $\varphi$ belongs to  $\Stab(\chi_1)$, then the matrix  $C$ commutes with $B$, otherwise it holds $CB=B^2 C$. 
We see that the matrix $C$ must have  the form 
\[
C=\begin{pmatrix} \lambda & -\mu \\ \mu & \lambda \end{pmatrix} \qquad \makebox{or} \qquad 
C=\begin{pmatrix} \lambda & \mu \\ \mu & -\lambda \end{pmatrix}.
\]
In complex coordinates, $C(z)= cz$ or  $C(z)= c\overline{z}$, where $c:=\lambda+ \sqrt{-1} \mu$. 
If there exists $D \in \GL(6,\mathbb R)$ such that $A = \diag(C,D) \in \mathcal N_{\mathbb R}(\Lambda,\Lambda')$, then $C \mathbb Z[\zeta_3] =  \mathbb Z[\zeta_3]$, because $\Lambda_i=\mathbb Z[\zeta_3] \oplus \Lambda_i^\natural$.
Hence, the scalar $c$ must be a unit in 
$\mathbb Z[\zeta_3]$. 
In total, there are  $12$ such matrices, they  
 form a dihedral group $\mathcal D_6$.  
\end{rem}


To determine the possibilities for $D$ is more involved. 
According to Remark \ref{rem1}, we obtain for each $\varphi \in \Stab(\chi_{1 \mathbb R})$ a matrix $D_{\varphi}$, which is either $\mathbb C$-linear or $\mathbb C$-antilinear and unique up to a nonzero constant.
Hence, $\varphi $ determines a  class $[D_{\varphi}] \in  \PGL(3,\mathbb C) \rtimes \mathbb Z_2$. It is clear that $[D_{\varphi}]\cdot [D_{\varphi'}]=[D_{\varphi\circ \varphi'}]$. In other words, we obtain 
 a semi-projective representation 
\[
\Xi \colon \Stab(\chi_{1 \mathbb R}) \to  \PGL(3,\mathbb C) \rtimes \mathbb Z_2, \quad  \varphi \mapsto [D_{\varphi}]. 
\]
It is  faithful because $\rho_{\mathbb R}$ is faithful, i.e., any class in the image of $\Xi$ corresponds to a unique automorphism $\varphi$. 
The normalizer  $\mathcal N_{\mathbb R}(\Lambda)$ has therefore a fairly simple description:

\begin{cor}	
The normalizer $\mathcal N_{\mathbb R}(\Lambda)$ is  equal to the fiber product 
\[
\mathcal N_{\mathbb R}(\Lambda) = \mathcal D_6 \times_{\mathbb Z_2} \lbrace D  \in  \GL(3,\mathbb C) \rtimes \mathbb Z_2 ~ 
\big\vert ~ [D]  \in \im(\Xi), ~ D \Lambda^\natural = \Lambda^\natural \rbrace 
\] 
defined via the  natural homomorphisms $\mathcal D_6 \to \mathcal D_6/\langle -\zeta_3\rangle \simeq  \mathbb Z_2$ and 
\[
 \lbrace D  \in  \GL(3,\mathbb C) \rtimes \mathbb Z_2 ~ 
\big\vert ~ [D]  \in \im(\Xi), ~ D \Lambda^\natural = \Lambda^\natural \rbrace \to \Stab(\chi_{1\mathbb R})/\Stab(\chi_1)\simeq \mathbb Z_2.
\]
\end{cor}

Our next goal is to explicitly describe the sets $\mathcal N_{\mathbb R}(\Lambda,\Lambda')$ and $\mathcal N_{\mathbb C}(\Lambda,\Lambda')$. We begin with the case $\Lambda=\Lambda'$ and determine $\mathcal N_{\mathbb R}(\Lambda)$ by finding generators of the group 
\[
 \lbrace D  \in  \GL(3,\mathbb C) \rtimes \mathbb Z_2 ~ 
\big\vert ~ [D]  \in \im(\Xi), ~ D \Lambda^\natural = \Lambda^\natural \rbrace \leq \GL(3,\mathbb C) \rtimes \mathbb Z_2. 
\]
The starting point is the following lemma:

\begin{lemma}\label{TableSemi}
The group $\Stab(\chi_{1 \mathbb R})\leq \Aut(\Heis(3))$ is generated by  four automorphisms $\varphi_1, \ldots, \varphi_4$. These automorphisms and 
representatives $D_{\varphi_i}$ of the classes $\Xi(\varphi_i)$ are listed in the table below: 

\bigskip
\begin{center}
		\begin{tabular}{|c|c|c|c|c|}
		\hline
			$i$  & $1$  & $2$  & $3$ & $4$ \\ 
			\hline \hline &&&& \\
			$\varphi_i(g)$  & $gk$  & $g$  & $g^2$ & $g$ \\ 
			&&&& \\
				$\varphi_i(h)$  & $h$  & $g^2hk^2$  & $h^2$  & $h^2$ \\ 
				&&&&  \\
$D_{\varphi_i}$ & $\begin{pmatrix} \zeta_3 & 0 & 0 \\ 0 & \zeta_3^2 & 0 \\ 0 & 0 & 1  \end{pmatrix} $  & 
$\frac{2+\zeta_3}{3}\cdot \begin{pmatrix}  1 & \zeta_3^2 & \zeta_3^2 \\ \zeta_3^2 & 1 & \zeta_3^2 \\ \zeta_3^2 & \zeta_3^2& 1    \end{pmatrix}$  & 
$\begin{pmatrix}  1 & 0 & 0 \\ 0 & 0 & 1 \\ 0 & 1 & 0 \end{pmatrix}$ &  $ (z_2,z_3,z_4)\mapsto(\overline{z}_2, \overline{z}_3, \overline{z}_4)$ \\
			&&&& \\ 
			\hline
		\end{tabular}
	\end{center}
	
\bigskip	
The representatives $D_{\varphi_i}$ are chosen such that $D_{\varphi_i}\cdot \Lambda^\natural=\Lambda^\natural$, independent of $\Lambda = \Lambda_1$ or $\Lambda_2$.
\end{lemma}

\begin{proof}
We first determine the normal subgroup $\Stab(\chi_1, \chi_3)\trianglelefteq \Stab(\chi_{1\mathbb R})$. 
Since $\Aut(\Heis(3))\simeq \AGL(2,\mathbb F_3)$ acts transitively on the eight non-trivial  complex $\Heis(3)$-characters of degree one, we have 
$\vert \Stab(\chi_1)\vert =432/8=54$. 
The automorphism $\varphi$ defined by $\varphi(g)=g^2$ and $\varphi(h)=h$ belongs to 
$\Stab(\chi_1)$. It maps  $\chi_3$ to $\overline{\chi_3}$, which is the second of the two irreducible complex characters of degree $3$. 
This implies  that  $\Stab(\chi_1, \chi_3)$  has order $27$. The elements $\varphi_1$ and $\varphi_2$ of this group have order 3 and do not commute, thus $\langle \varphi_1,\varphi_2\rangle = \Stab(\chi_1, \chi_3)$.\\
Similarly, we observe $\vert \Stab(\chi_{1\mathbb R})\vert =108=4\cdot 27$. 
This group 
contains the two involutions  $\varphi_3$ and $\varphi_4$ that commute. Therefore, these elements generate a complement of 
$ \Stab(\chi_1, \chi_3)\trianglelefteq \Stab(\chi_{1\mathbb R})$.\\
To find the representatives $D_{\varphi_i}$, we first remark that $\varphi_1,\varphi_2$ and $ \varphi_3$ belong to  $\Stab(\chi_3)$, thus $D_{\varphi_1}$, $D_{\varphi_2}$ and  $D_{\varphi_3}$ are $\mathbb C$-linear, see Remark \ref{rem1}. We compute them as  solutions of the equations
\[
D  \rho_3(u) =(\rho_3 \circ \varphi_i)(u) D, \qquad u=g, h.
\]
Here, $\rho_3$ is  the representation defined in \ref{rho3}. 
On the other hand, $\varphi_4$ exchanges $\chi_3$ and $\overline{\chi}_3$, which implies  that $D_{\varphi_4}$ is antilinear. 
It is easy to check that $D_{\varphi_4}(z_2,z_3,z_4)=(\overline{z}_2, \overline{z}_3, \overline{z}_4)$ solves the equation 
$D  \rho_{3\mathbb R}(u) =(\rho_{3\mathbb R} \circ \varphi_4)(u) D$ for $u=g$ and $h$.
\end{proof}

\noindent 
Up to now, it is not even clear that $\mathcal N_{\mathbb R}(\Lambda)$ is finite. 
The next lemma assures this property.

\begin{lemma}\label{possiblereps}
Suppose that a representative $D_{\varphi}$ of $\Xi(\varphi)$ is chosen such that 
$D_{\varphi} \Lambda^\natural=\Lambda^\natural$, then  all other representatives with this  property differ by a unit of 
$\mathbb Z[\zeta_3]$. 
\end{lemma}

\begin{proof}
	Let $\mu \in \mathbb C^{\ast}$ with $\mu D_{\varphi} \Lambda^\natural = \Lambda^\natural$. Since $D_{\varphi} \Lambda^\natural=\Lambda^\natural$, 
	we have 
	$\mu \Lambda^\natural=\mu D_{\varphi} \Lambda^\natural= \Lambda^\natural$.
	In particular, $\mu \cdot (1,0,0) \in \Lambda^\natural$, which implies
	$\mu \in \mathbb Z[\zeta_3]$. On the other hand, the equation $\mu \Lambda^\natural = \Lambda^\natural$  is  equivalent to 
	$\mu^{-1} \Lambda^\natural = \Lambda^\natural$, and  we  conclude that  
	$\mu^{-1} \in \mathbb Z[\zeta_3]$. 
\end{proof}

\begin{prop}\label{Di}
The semi-linear transformations $D_{\varphi_1}, \ldots, D_{\varphi_4}$ in the table of Lemma \ref{TableSemi} generate the group 
\[
\lbrace D  \in  \GL(3,\mathbb C) \rtimes \mathbb Z_2 ~ 
\big\vert ~ [D]  \in \im(\Xi), ~ D \Lambda^\natural = \Lambda^\natural \rbrace \leq \GL(3,\mathbb C) \rtimes \mathbb Z_2. 
\] 
\end{prop}

\begin{proof}
By construction,  $\langle D_{\varphi_1}, \ldots, D_{\varphi_4} \rangle$ is contained in
$\lbrace D  \in  \GL(3,\mathbb C) \rtimes \mathbb Z_2 ~ 
\big\vert ~ [D]  \in \im(\Xi_{\mathbb R}), ~ D \Lambda^\natural = \Lambda^\natural \rbrace$. The latter  
has at most $\vert\Stab(\chi_{1\mathbb R})\vert \cdot 6 =648$ elements, thanks to Lemma 
\ref{possiblereps}. 
On the other hand, a MAGMA computation shows that $\langle D_{\varphi_1}, \ldots, D_{\varphi_4} \rangle$ has also $648$ elements. 
Hence, both groups are equal. 
 \end{proof}

\begin{cor}\label{thenormal}
For both $\Lambda= \Lambda_1$ and $\Lambda_2$, 
the normalizers   $\mathcal N_{\mathbb R}(\Lambda)$ and $\mathcal N_{\mathbb C}(\Lambda)$ are
\[
\mathcal N_{\mathbb R}(\Lambda)=\mathcal D_6 \times_{\mathbb Z_2} \langle D_{\varphi_1}, \ldots, D_{\varphi_4} \rangle \quad \makebox{and} \quad  
\mathcal N_{\mathbb C}(\Lambda)=\langle -\zeta_3 \rangle \times \langle D_{\varphi_1}, D_{\varphi_2} \rangle.
\]
\end{cor}

\begin{proof}
Only the claim for $\mathcal N_{\mathbb C}(\Lambda)=\mathcal N_{\mathbb R}(\Lambda) \cap \GL(4,\mathbb C)$ needs to be justified. 
According to Remark  \ref{rem1}, a matrix $A=A_{\varphi}$ belongs to $\mathcal N_{\mathbb C}(\Lambda)$ if and only if 
$\varphi \in  \Stab(\chi_1, \chi_3)$. For this reason 
\[
 \mathcal N_{\mathbb C}(\Lambda) =\langle -\zeta_3 \rangle \times \lbrace D  \in  \GL(3,\mathbb C) ~ 
\big\vert ~ [D]  \in \Xi(\Stab(\chi_1, \chi_3)), ~ D \Lambda^\natural = \Lambda^\natural \rbrace.
\]
Since  $\Stab(\chi_1, \chi_3)$ is generated by $\varphi_1$ and $\varphi_2$, the 
image $\Xi\big(\Stab(\chi_1, \chi_3)\big)$ is generated by the projective transformations $[D_{\varphi_1}]$ and $[D_{\varphi_2}]$. 
Thanks to Lemma \ref{possiblereps}, we conclude that 
\[
\lbrace D  \in  \GL(3,\mathbb C) ~ 
\big\vert ~ [D]  \in \Xi(\Stab(\chi_1, \chi_3)), ~ D \Lambda^\natural = \Lambda^\natural \rbrace = \langle D_{\varphi_1}, D_{\varphi_2} \rangle,
\]  
in analogy to Proposition \ref{Di}.
 \end{proof}

\begin{prop}\label{nothomeo}
The set  $\mathcal N_{\mathbb R}(\Lambda_1,\Lambda_2)$ is empty. In particular, hyperelliptic fourfolds with holonomy $\Heis(3)$ corresponding to different lattices are topologically distinct.
\end{prop}

\begin{proof}
If $\mathcal N_{\mathbb R}(\Lambda_1,\Lambda_2)$ is not empty, let $A = \diag(C,D) \in 	\mathcal N_{\mathbb R}(\Lambda_1,\Lambda_2)$, so that $A \cdot \Lambda_1 = \Lambda_2$.

It holds $D \cdot \Lambda_1^\natural =\Lambda_2^\natural$. Since 
the class of  $D$ belongs to $\im(\Xi) \leq \PGL(3,\mathbb C) \rtimes \mathbb Z_2 $, there is a constant 
$\lambda \in \mathbb C^{\ast}$ such that $\lambda D \in \langle D_{\varphi_1}, \ldots, D_{\varphi_4} \rangle$.  This implies $\lambda \Lambda_2^\natural=\lambda D \Lambda_1^\natural = \Lambda_1^\natural$ and in particular, 
$\lambda e_1 \in \Lambda_1^\natural$, which shows that $\lambda \in \mathbb Z[\zeta_3]$.  
On the other hand, we have  $\lambda^{-1} \Lambda_1^\natural =\Lambda_2^\natural$, which implies  $\lambda^{-1} \in \mathbb Z[\zeta_3]$, and we conclude that  $\lambda$
is a unit of $\mathbb Z[\zeta_3]$. Hence, $D$ itself is an  element in $\langle D_{\varphi_1}, \ldots, D_{\varphi_4} \rangle$. This leads to the  contradiction  $D \Lambda_1^\natural =\Lambda_1^\natural \subsetneq \Lambda_2^\natural$. 
\end{proof}

In order to decide whether two given translation parts corresponding to free actions are in the same $\mathcal N_{\mathbb C}(\Lambda)$- or $\mathcal N_{\mathbb R}(\Lambda)$-orbit, we determine the possible $d$'s, which define the coboundaries  in Condition (2) of Remark \ref{Con1and2} (see also Remark \ref{cocycle-rem} (2)). Here, we use the notation $d = (d_1, \ldots, d_4) \in \CC^4 \simeq \RR^8$, as in Remark \ref{FundBie} \ref{FundBie1}.

\begin{prop} \label{TParts-Heis3}
Let $X$ and $X'$ be quotients of $T= \mathbb C^4/\Lambda$ with respect to the rigid and free actions $\Phi, \Phi' \colon \Heis(3) \to \Aut(T)$, where $\Lambda = \Lambda_1$ or $\Lambda_2$ and  $\Lambda^\natural := \Lambda \cap (\lbrace 0 \rbrace \times \mathbb C^3)$. Suppose that $f \colon X \to X'$ is a \emph{diffeomorphism} induced by the affinity $\alpha(x) = Ax+d$. Then:
\begin{enumerate}
	\item The element $(d_2,d_3,d_4)$ is contained in the kernel of $(\zeta_3-1) \colon \mathbb C^3/\Lambda^\natural \to \mathbb C^3/\Lambda^\natural$.
	\item The first coordinate $d_1$ of $d$ is contained in the kernel of the map $3(\zeta_3-1) \colon E \to E$.
\end{enumerate}
\end{prop}

\begin{proof}
(1) Since the center $Z(\Heis(3))=\langle k \rangle $ is characteristic, we have $\varphi(k)=k$ or $k^2$ for all automorphisms $\varphi$ of $\Heis(3)$.
By definition, $\Phi'(k^j)$ acts on the last three coordinates by multiplication by $\zeta_3^j$. Spelling out condition (2) of Remark \ref{Con1and2} for $u = k$ thus yields 
\[
(\zeta_3^j-1)(d_2,d_3,d_4)=0 \in \mathbb C^3/\Lambda^\natural. 
\]
The statement follows, because $\ker(\zeta_3-1)= \ker(\zeta_3^2-1)$.

(2) We use Remark \ref{Con1and2} (2) for $u=\varphi^{-1}(h)$. Since the right-hand side is $3$-torsion by Lemma \ref{elementary-properties}, we conclude
$3(\zeta_3-1)d_1  \in \mathbb Z[\zeta_3]$. 
\end{proof}

\begin{prop} \label{Heis-Class}
There are precisely four biholomorphism classes of rigid hyperelliptic fourfolds with holonomy $\Heis(3)$. They are topologically distinct. 
\end{prop}

\begin{proof} 
We use MAGMA to determine for each lattice 
$\Lambda=\Lambda_1$ and  $\Lambda_2$  the set of all possible  translation parts $\tau$ of free  actions  in standard form, as described in Remark \ref{ActHeis3}. They represent the special cohomology classes in 
$H^1(\Heis(3), \mathbb C^4/\Lambda)$. For every pair of special cocycles $\tau, \tau'$, we check whether there is a matrix $A \in \mathcal N_{ \mathbb C}(\Lambda)$ and a vector $d$ according to Proposition \ref{TParts-Heis3} such that 
\begin{align*}
	(\rho(u) - I_4)d = A \ast \tau(u) - \tau'(u) 
\end{align*}
for $u = g,h$. Our code finds two orbits for each of the two lattices $\Lambda_1$ and $\Lambda_2$. This computation and Proposition  \ref{nothomeo} imply that there are exactly four biholomorphism classes and at least two different topological types. We claim that the two biholomorphism classes obtained for each lattice are topologically distinct. To verify the claim, we check with MAGMA that the two respective cohomology classes in $H^1(\Heis(3), \mathbb C^4/\Lambda_i)$ belong to different $\mathcal N_{\mathbb R}(\Lambda_i)$-orbits. 
\end{proof}

\subsection{The Case $\mathbb Z_3^2$ }

\begin{prop}
Let $K$ be one of the twelve kernels from Remark  \ref{13Shaolin}, then, 
\[
\mathcal N_{\mathbb C}(\Lambda_K) = \lbrace A \in N_{\Aut_0(E^4)}(\mathbb Z_3^2) ~\big\vert ~AK=K\rbrace.  
\]
\end{prop}
\begin{proof}
As we have seen in the proof of Proposition \ref{Z3^2-normalizer}, each 
$A \in  N_{\Aut_0(E^4/K)}(\mathbb Z_3^2) $ also defines  an automorphism of $E^4$ and therefore stabilizes  $K$. 
The other inclusion is clear. 
\end{proof}

We determine the possible translation parts of potential biholomorphisms, mimicking the statement in the Heisenberg case (Proposition \ref{TParts-Heis3}).

\begin{prop} \label{TParts-Z3^2}
Let $X$ and $X'$ be quotients of $T= E^4/K$ with respect to the rigid and free actions $\Phi, \Phi' \colon \ZZ_3^2 \to \Aut(T)$ in standard form. Suppose that $f \colon X \to X'$ is a biholomorphism induced by the affinity $\alpha(x) = Ax+d$. Then:
\begin{enumerate}
	\item The element $(d_2,d_3,d_4)$ is contained in the kernel of $(\zeta_3-1) \colon \mathbb C^3/p(\Lambda_K) \to \mathbb C^3/p(\Lambda_K)$, where $p \colon \mathbb C^4 \to \mathbb C^3$ is the projection onto the last three coordinates. 
	\item The first coordinate $d_1$ of $d$ is contained in the kernel of the map $3(\zeta_3-1) \colon E \to E$.
\end{enumerate}
\end{prop}

\begin{proof}
(1) The proof of Proposition \ref{Z3^2-normalizer} shows that $\varphi(k) = k$, where $\varphi$ is the unique automorphism such that $A \rho A^{-1} = \rho \circ \varphi$ (see Remark \ref{cocycle-rem} (3)). We can thus argue as in the proof of Proposition \ref{TParts-Heis3}.\\
(2) This is similar to the proof of Proposition \ref{TParts-Heis3} (2).
\end{proof}

\begin{prop}\label{oneBihol}
For each of the twelve kernels $K_i$  from Remark  \ref{13Shaolin}, there exists one and only one 
biholomorphism class of a rigid hyperelliptic fourfold $X_i$. 
\end{prop}

\begin{proof} 
In analogy to Proposition \ref{Heis-Class}, we use MAGMA to verify that $\mathcal N_{\mathbb C}(\Lambda_{K_i})$ acts transitively on the special cohomology classes in $H^1(\mathbb Z_3^2, E^4/K_i)$ for each kernel $K_i$.	Recall that there are no rigid and free $\mathbb Z_3^2$-actions on $E^4/K_{\exc}$ (i.e., no special cohomology classes in $H^1(\mathbb Z_3^2, E^4/K_{\exc})$), see Remark \ref{13Shaolin}.
 \end{proof}

Using Proposition \ref{oneBihol}, the diffeomorphism problem can be reformulated in the following way: 

\begin{cor} \label{Z3^2-DiffeoCor}
The fourfolds $X_i$ and $X_j$  corresponding to  kernels  $K_i$ and $K_j$ 
are diffeomorphic if and only 
if $ \mathcal N_{\mathbb R}(\Lambda_{K_i}, \Lambda_{K_j})$ is not empty.  
\end{cor}

\begin{proof}
Assume that $A\in \mathcal N_{\mathbb R}(\Lambda_{K_i}, \Lambda_{K_j})$ and let $\Phi_i$ be a free holomorphic action giving  
$X_i$, then, $\psi:=A \Phi_i A^{-1}$ is a free action on $E^4/K_j$ with linear part $A\rho_{\mathbb R} A^{-1} = \rho_{\mathbb R}\circ \varphi_A$. In particular, $\psi$ is holomorphic and rigid. The matrix $A$ induces a diffeomorphism between $X_i$ and the quotient $Z$ with respect to $\psi$. However, $Z$ and $X_j$ are biholomorphic since  by Proposition \ref{oneBihol}, there is precisely one biholomorphism class corresponding to the kernel $K_j$. The converse is obvious. 
\end{proof}

We therefore need to determine the $i \neq j$ for which $\mathcal N_{\mathbb R}(\Lambda_{K_i}, \Lambda_{K_j})$ is non-empty.

\begin{rem}
Let $A$ be a matrix in $\mathcal N_{\mathbb R}(\Lambda_{K_i}, \Lambda_{K_j})$, 
then, there exists an automorphism $\varphi \in \Aut(\mathbb Z_3^2)$ such that 
\[
A\rho_{\mathbb R}A^{-1}=\rho_{\mathbb R}\circ \varphi.
\]
Note that the representation  $\rho_{\mathbb R}$ is the direct sum of all four irreducible $2$-dimensional real representations 
of $\mathbb Z_3^2$: 
\[
\rho_{\mathbb R}(a,b)=\diag(B^a, \ B^b, \ B^{2a+b}, \ B^{a+b}),  \quad \makebox{where} \quad  
B=
-\frac{1}{2}
\begin{pmatrix} 1 &  \sqrt{3}  \\ -\sqrt{3} &1 \end{pmatrix}.
\]
By Schur's lemma, $A$ is, up to a permutation, a block diagonal matrix of $2\times 2$ blocks. 
In analogy to Remark \ref{Antilinear}, they are $\mathbb C$-linear or antilinear, i.e., of the form 
\[
z \mapsto cz \qquad \makebox{or} \qquad z \mapsto c\overline{z}. 
\]
Since  $AK_i=K_j$, we conclude as in the proof of  Proposition \ref{Z3^2-normalizer} that $A\cdot \mathbb Z[\zeta_3]^4=\mathbb Z[\zeta_3]^4$. This in turn implies that the scalars 
$c$ are units in $\mathbb Z[\zeta_3]$. 
In other words, 
\[
\mathcal N_{\mathbb R}(\Lambda_{K_i}, \Lambda_{K_j})= \lbrace A \in N_{\Aff_0(E^4)}(\mathbb Z_3^2) ~\big\vert ~AK_i=K_j\rbrace.  
\]
The group $N_{\Aff_0(E^4)}(\mathbb Z_3^2)$ has $\vert \Aut(\mathbb Z_3^2)\vert \cdot 6^4= 62208$ elements and is generated by  the maps
\begin{itemize}
\item $M_1(z_1,z_2,z_3,z_4):=(\overline{z}_1,z_2,z_4,z_3)$,
\item $M_2(z_1,z_2,z_3,z_4):=(z_2,\overline{z}_3,z_4,z_1)$ and
\item $M_3(z_1,z_2,z_3,z_4):=(-\zeta_3z_1,z_2,z_3,z_4)$.
\end{itemize}
Note that  $N_{\Aff_0(E^4)}(\mathbb Z_3^2)$ is a subgroup of $O(8)$. In particular, the determinant of any of its elements is $\pm1$. 
\end{rem}

\begin{cor}
Let $\mu_m(K)$ be the number of elements in $K$ with exactly $m$ non-zero entries. Then, 
$\mathcal N_{\mathbb R}(\Lambda_{K_i}, \Lambda_{K_j})$ is empty if 
$\mu_m(K_i) \neq  \mu_m(K_j)$ for some $m$. 
\end{cor}

\begin{proof}
According to the previous discussion, any $A \in \mathcal N_{\mathbb R}(\Lambda_{K_i}, \Lambda_{K_j})$ can be regarded as an isomorphism 
$A\colon K_i \to K_j$ of finite groups. The functions $\mu_m$ are  invariant because $A$ is the direct sum of four maps of the form $z \mapsto cz$ or  $z \mapsto c\overline{z}$, up to a permutation. 
\end{proof}

Going through the list of all kernels $K_i$  from Remark  \ref{13Shaolin}, we obtain:

\begin{prop} \label{Z3^2-diffeo}
 The  set $\mathcal N_{\mathbb R}(\Lambda_{K_i}, \Lambda_{K_j})$ is empty for all  $1 \leq i < j \leq 12$, except in the following cases:
\begin{center}	\bgroup\def\arraystretch{1.3}
	\begin{tabular}{|c|c|c|} \hline
		$i$ & $j$ & Size of $\mathcal N_{\mathbb R}(\Lambda_{K_i}, \Lambda_{K_j})$ \\ \hline \hline
		$2$ & $4$ & $5184$ \\ \hline 
		$3$ & $5$ & $3888$ \\ \hline 
		$7$ & $8$ & $3888$  \\ \hline 
		$10$ & $11$ & $1296$ \\ \hline 
	\end{tabular} \egroup
\end{center}
\medskip
Furthermore, in  each set $\mathcal N_{\mathbb R}(\Lambda_{K_i}, \Lambda_{K_j})$ contained in the table, half of the 
elements are orientation-preserving, the other half orientation-reversing.
%
\end{prop}

\begin{prop} \label{Z3^2-pi1}
A hyperelliptic fourfold  $X$  whose fundamental group is isomorphic to the fundamental group of a rigid hyperelliptic fourfold  with holonomy $\mathbb Z_3^2$ is rigid. 
\end{prop}

\begin{proof}
According to Corollary \ref{GeometricBieber} (2), the isomorphism of fundamental groups induces an affine diffeomorphism of the manifolds.
Note that the complex holonomy representation $\rho' \colon \mathbb Z_3^2 \to \GL(4,\mathbb C)$ of $X$ cannot contain complex conjugate sub-representa\-tions, 
 otherwise 
the decomplexification  $\rho'_{\mathbb R}$ consists of at most three distinct non-trivial irreducible real representations -- a contradiction because 
$\rho'_{\mathbb R}$ is equivalent to $\rho_{\mathbb R}(a,b)=\diag(B^a, \ B^b, \ B^{2a+b}, \ B^{a+b})$ up to an automorphism of $\mathbb Z_3^2$ (see Remark \ref{Con1and2}). The latter however contains four non-trivial distinct irreducible real representations. According to Corollary \ref{ConjRig}, the manifold $X$ is rigid. 
\end{proof}

\section{Proof of the Second Main Theorem} 

In this section, we summarize the proof of our second main result. 

\begin{proof}[Proof of Theorem \ref{Mani}]
In Remark \ref{analyticreps}, the analytic representations for both $G = \mathbb Z_3^2$ and $\Heis(3)$ are described. Proposition \ref{isog} shows that a hyperelliptic fourfold with such a holonomy representation is finitely covered by a product of four Fermat elliptic curves. \\
(1) It was proved in Proposition \ref{oneBihol} that there are exactly twelve biholomorphism classes of rigid hyperelliptic fourfolds with holonomy $\mathbb Z_3^2$. We obtain the translation parts $\tau_i$ listed in the theorem using our MAGMA code. The statement regarding the diffeomorphism types follows from Corollary \ref{Z3^2-DiffeoCor} and the computation in Proposition \ref{Z3^2-diffeo}. Proposition \ref{Z3^2-pi1} shows that a hyperelliptic fourfold whose fundamental group is isomorphic to the fundamental group of a rigid hyperelliptic fourfold with holonomy $\mathbb Z_3^2$ is rigid and therefore biholomorphic to one of the $X_i$. \\
(2) The classification follows from Proposition \ref{nothomeo} and Proposition \ref{Heis-Class}. The listed cocycles are a by-product of our MAGMA computation. In Proposition \ref{He3-rigid}, we explained that any  hyperelliptic fourfold with holonomy $\Heis(3)$ is rigid.
\end{proof}

%
%
%
%
%
%
%


\begin{thebibliography}{xxxxxxxx}
	
	\bibitem[BC18]{Bauer-Catanese-rigid} \textsc{I. Bauer, F. Catanese}: On rigid compact complex surfaces and manifolds, Adv. Math. 333 (2018), 620--669.
	
	\bibitem[BCP97]{MAGMA} \textsc{W. Bosma, J. Cannon, C. Playoust}: The Magma algebra system. I. The user language, J. Symbolic Comput., 24 (1997), 235–265. 

		\bibitem[BdF08]{bdf} \textsc{G.  Bagnera, M. de Franchis}: 
	Le superficie algebriche le quali ammettono una rappresentazione parametrica mediante funzioni iperellittiche di due argomenti, 
	Mem. di Mat. e di Fis. Soc. It. Sc. (3) 15, 253--343 (1908).

	\bibitem[BL04]{Birkenhake-Lange} \textsc{C. Birkenhake, H. Lange}: Complex abelian varieties. Second edition. Grundlehren der Mathematischen Wissenschaften, 302. Springer-Verlag, Berlin (2004).


		\bibitem[BG21]{Bauer-Gleissner-2} \textsc{I. Bauer, C. Gleissner}: Towards a Classification of Rigid Product Quotient Varieties of Kodaira Dimension $0$. Bollettino dell'Unione Matematica Italiana, Springer--Verlag, (2021).
	
	\bibitem[BBP20]{BBP} \textsc{C. B\"ohning, H. C. von Bothmer, R. Pignatelli}: A rigid, not infinitesimally rigid surface with $K$ ample. Bollettino dell'Unione Matematica Italiana, Springer--Verlag,  (2020).
	
	\bibitem[BP21]{Bauer-Pignatelli} \textsc{I. Bauer, R. Pignatelli}: Rigid but not infinitesimally rigid compact complex manifolds. Duke Math. J. 170 (2021), no. 8, 1757--1780.

\bibitem[C86]{Charlap} \textsc{L. Charlap}: Bieberbach groups and flat manifolds, Universitext, Springer--Verlag, New York (1986).

	\bibitem[Cat11]{Catanese-Guide} \textsc{F. Catanese:} A Superficial Working Guide to Deformations and Moduli, Advanced Lectures in Mathematics, vol. XXVI, Handbook of Moduli, vol. III (International Press 2011), pp. 161--216.

	
	\bibitem[CD17]{CD} \textsc{F. Catanese, A. Demleitner}: Rigid Group Actions on Complex Tori are Projective (after Ekedahl). Commun. Contemp. Math. 22 (2020), no. 7.
	
	\bibitem[CD18-2]{CD-2} \textsc{F. Catanese, A. Demleitner}: The classification of hyperelliptic threefolds. Groups Geom. Dyn. 14 (2020), no. 4, 1447-1454.
	
	
	\bibitem[Dem22]{Demleitner-thesis} \textsc{A. Demleitner}: The Classification of Hyperelliptic Groups in Dimension $4$. Preprint (2022). Available at \href{https://arxiv.org/abs/2211.07998}{https://arxiv.org/abs/2211.07998}.
	
	\bibitem[ES09]{Enr-Sev} \textsc{F. Enriques, F. Severi}: M\'emoire sur les surfaces hyperelliptiques. Acta Math. 32, 283-392 (1909) and 33, 321-403 (1910).
	

	\bibitem[HL21]{HalendaLutowski} \textsc{M. Ha\l enda, R. Lutowski}: Symmetries of complex flat Manifolds.
	Preprint (2021). Available at \href{https://arxiv.org/abs/1905.11178}{https://arxiv.org/abs/1905.11178}.
	
	\bibitem[HS86]{HillerSah1} \textsc{H. Hiller, C.-H. Sah}: Holonomy of flat manifolds with $b_1=0$. Quart. J. Math. Oxford Ser. (2) 37 (1986), no. 146, 177–187.
	
	\bibitem[HMSS87]{HillerSah2} \textsc{H. Hiller, Z. Marciniak, C.-H. Sah, A. Szczepa\'nski} Holonomy of flat manifolds with $b_1=0$. II. Quart. J. Math. Oxford Ser. (2) 38 (1987), no. 150, 213–220. 
	
\bibitem[K05a]{Kodaira} \textsc{K. Kodaira}: Complex manifolds and deformation of complex
structures.  Classics in Mathematics.
Springer-Verlag, Berlin (2005).	

\bibitem[K05b]{KM}  \textsc{K. Kodaira, J. Morrow}:  Complex manifolds. Holt, Rinehart and
Winston, Inc., New York--Montreal, Que.--London (1971).	

	\bibitem[La01]{Lange} \textsc{H. Lange}: Hyperelliptic varieties. Tohoku Math. J. (2) 53 (2001), no. 4, 491--510.
	
	\bibitem[S77]{Serre}  \textsc{J.-P. Serre}: Linear representations of finite groups. Graduate Texts in Mathematics, Vol. 42. Springer-Verlag, New York--Heidelberg (1977).
	
	\bibitem[S12]{Szczepanski} \textsc{A. Szczepa\'nski}: Geometry of Crystallographic Groups. Algebra and Discrete Mathematics, 4. World Scientific Publishing Co. Pte. Ltd., Hackensack, NJ (2012).

	\bibitem[UY76]{Uchida-Yoshihara} \textsc{K. Uchida, H. Yoshihara}: Discontinuous groups of affine transformations of $\CC^3$. Tohoku Math. J. (2) 28 (1976), no. 1, 89--94.

\end{thebibliography}
\end{document}